\RequirePackage{fix-cm}

\documentclass[smallextended]{svjour3}       
\smartqed  
\usepackage{graphicx,amssymb,latexsym}
\begin{document}

\title{The stratification by rank for homogeneous polynomials with border rank 5 which essentially depend on 5 variables\thanks{Partially supported by
MIUR and GNSAGA of INdAM}}

\titlerunning{Border rank $5$}        

\author{Edoardo Ballico}


\institute{E. Ballico\at
            Dept. of Mathematics, University of Trento, 38123 Trento (Italy) \\
              Tel.: 39+0461281646\\
              Fax: 39+0461281624
              \email{ballico@science.unitn.it}}

              \date{Received: date / Accepted: date}

\maketitle

\begin{abstract}
We give the stratification by the symmetric tensor rank of all degree $d \ge 9$ homogeneous polynomials
with border rank $5$ and which depend essentially on at least 5 variables, extending previous works
(A. Bernardi, A. Gimigliano, M. Id\`{a}, E. Ballico) on lower border ranks. For the polynomials
which depend on at least 5 variables only 5 ranks are possible: $5$, $d+3$, $2d+1$, $3d-1$, $4d-3$, but each of the ranks
$3d-1$ and $2d+1$ is achieved in two geometrically different situations. These ranks are uniquely determined by a certain degree
5 zero-dimensional scheme $A$ associated to the polynomial. The polynomial $f$ depends essentially on at least 5 variables
if and only if $A$ is linearly independent (in all cases $f$ essentially depends on exactly 5 variables). The polynomial has rank $4d-3$ (resp $3d-1$, resp. $2d+1$, resp. $d+3$, resp. $5$)
if $A$ has $1$ (resp. $2$, resp. $3$, resp. $4$, resp. $5$) connected components. The assumption $d\ge 9$ guarantees that each polynomial
has a uniquely determined associated scheme $A$. In each case we describe the dimension of the families of the
polynomials with prescribed rank, each irreducible family being determined by the degrees
of the connected components of the associated scheme $A$.
\keywords{symmetric tensor rank\and symmetric rank\and border rank\and cactus rank}
 \subclass{14N05}
\end{abstract}

\section{Introduction}
\label{intro}
Let $\mathbb {C}[x_0,\dots ,x_m]_d$ be the set of all homogeneous degree $d$ polynomials in the variables $x_0,\dots ,x_m$ with complex coefficients.
For any $f\in \mathbb {C}[x_0,\dots ,x_m]_d\setminus \{0\}$ the {\it rank} of $f$ (or the {\it symmetric tensor rank} of $f$) is the minimal integer $r>0$ such
that $f = \sum _{i=1}^{r} \ell _i^d$ with $\ell _i\in \mathbb {C}[x_0,\dots ,x_m]_1$. These additive decompositions of $f$ as sums of $d$-powers of linear forms
are useful even to decompose symmetric tensors and hence they appear in some applications (\cite{ACCF}, \cite{cm}, \cite{DM}, \cite{Ch}, \cite{dLC}, \cite{McC}, \cite{Co1}, \cite{JS}, \cite{l}). A different notion is the notion of {\it border rank} of $f$ or approximate rank of $f$ (the minimal integer $r$ such that $f$ is the limit
of a family of homogeneous polynomials of rank $\le r$). Both notions may be translated in the setting of projective geometry and algebraic geometry
and it is in this setting that it was done the classification of all $f$ with border rank $\le 3$ (\cite[Theorems 32 and 37]{bgi}) and of border rank $4$ (\cite{bb2}).
In this paper we give the classification of all $f$ with border rank $5$ AND which depend essentially on at least 5 variables (in each case $P$ depends on exactly 5 variables, because it is contained in the degree $d$ Veronese embedding of a $4$-dimensional linear subspace of $\mathbb {P}^m$). This is a strong condition and in Remark \ref{e2e} and Proposition \ref{t1} we explain our knowledge of the rank for polynomials depending on fewer variables.

Now we explain the translation of the additive decomposition of homogeneous polynomials in terms of projective geometry.
For all positive integers $m, d$ let $\nu _d: \mathbb {P}^m\to \mathbb {P}^r$, $r:={m+d}\choose{m}-1$, denote the Veronese embedding induced by
the vector space $\mathbb {C}[x_0,\dots ,x_m]_d$. Set $X_{m,d}:= \nu _d(\mathbb {P}^m)$. For any $P\in \mathbb {P}^r$ the {\it rank} (or {\it symmetric tensor rank} or {\it symmetric rank})
$r_{m,d}( P)$ of $P$
is the minimal cardinality of a set $S\subset X_{m,d}$ such that $P\in \langle \nu _d(S)\rangle$, where $\langle \ \ \rangle$ denote the linear span.
For any integer $t\ge 1$ the $t$-secant variety $\sigma _t(X_{m,d})$ of $X_{m,d}$ is the closure in $\mathbb {P}^r$ of the union of all linear
spaces $\langle \nu _d(S)\rangle$ with $S\subset \mathbb {P}^m$ and $\sharp (S)\le t$
(to get the closure it is the same if we use the euclidean topology or the Zariski topology). This approximation makes sense for homogeneous polynomials, too:
the homogeneous polynomial $f$ is a limit of a family of homogeneous polynomials with rank $\le t$ if and only if its associated point
in $\mathbb {P}^r$ is contained in $\sigma _t(X_{m,d})$. Notice that ``~closure~'' in the definition of $\sigma _t(X_{m,d})$
is a good way to formalize the approximation by points with rank $\le t$. Set $\sigma _0(X_{m,d}) =\emptyset$.
The {\it border rank} $b_{r,m}( P)$ of $P$ is the minimal integer $t>0$ such that $P\in \sigma _t(X_{m,d})$, i.e. the only integer $b>0$ such that
$P\in \sigma _b(X_{m,d})\setminus \sigma _{b-1}(X_{m,d})$. 
The {\it cactus rank} (\cite{br}, \cite{rs}) (introduced in \cite{ik} with the name {\it scheme-rank}) $z_{m,d}( P)$ of $P$ is the minimal integer $z$ such that there is a zero-dimensional scheme $Z\subset \mathbb {P}^m$
with $\deg (Z)=z$ and $P\in \langle \nu _d(Z)\rangle$. We always have $z_{m,d}( P) \le r_{m,d}( P)$. If $d\ge b_{r,m}( P) -1$, then $z_{r,m}( P) \le b_{r,m}( P)$ and often equality
holds. We are interested in points with border rank $5$ and we assume that $d$ is not too low, e.g., we assume $d\ge 9$. Since
$d\ge 4$ and $b_{m,d}( P)=5$, we have $z_{m,d}( P)\le 5$ and hence we get at least one scheme $A$ with $\deg (A)\le 5$ and $P\in \langle \nu _d(A)\rangle$. Now assume $m\ge 4$, $d\ge 4$ and that $P\notin \langle \nu _d(M)\rangle$
for any linear subspace $M$ of $\mathbb {P}^m$ with $\dim (M) \le 3$, i.e. assume that the polynomial $f$ associated to $P$ does not depend
on at most 4 variables, up to a linear change of coordinates. Since $P\in \langle \nu _d(A)\rangle$, the condition ``~ $P\notin \langle \nu _d(M)\rangle$
for any linear subspace $M$ of $\mathbb {P}^m$ with $\dim (M) \le 3$~'' implies $\dim (\langle A\rangle )=4$, i.e. $A$ is linearly independent.
Assume that $A$ has $s\ge 1$ connected components, say $A = A_1\sqcup \cdots \sqcup A_s$. We say that $A$ has type $(s;b_1,\dots ,b_s)$. We have $5 =\deg (A) =b_1+\cdots +b_s$. In our case we have $z_{m,d}( P) =5$ if and only if $b_{m,d}( P) =5$ and the scheme $A$ evincing the cactus rank is unique (\cite[Corollary 2.7]{bgl} or \cite[Lemma 4.2 and text after Remark 2.7]{bb2}). Therefore if $d\ge 9$ we may define the type
of $P$ as the type of $A$. We prove the following result.

\begin{theorem}\label{i2}
Fix integers $m\ge 4$ and $d\ge 9$. Let $P\in \mathbb {P}^r$, $r={m+d}\choose{m}-1$, be a point with border rank $5$. Assume
that there is no 3-dimensional linear subspace $H\subset \mathbb {P}^m$ such that $P\in \langle \nu _d(H)\rangle$.
\begin{itemize}
\item[(i)] If $P$ has type $(1;5)$, then $r_{m,d}( P)=4d-3$.
\item[(ii)] If $P$ has type $(2;3,2)$, then $r_{m,d}( P)=3d-1$.
\item[(iii)] If $P$ has type $(2;4,1)$, then $r_{m,d}( P)=3d-1$.
\item[(iv)] If $P$ has type $(3;3,1,1)$, then $r_{m,d}( P)=2d+1$.
\item[(v)] If $P$ has type $(3;2,2,1)$, then $r_{m,d}( P)=2d+1$.
\item[(vi)] If $P$ has type $(4;2,1,1,1)$, then $r_{m,d}( P)=d+3$.
\item[(vii)] If $P$ has type $(5;1,1,1,1,1)$, then $r_{m,d}( P)=5$.
\end{itemize} 
\end{theorem}

In each case (i),\ldots ,(vii)  the description of each $A$ easily give the dimension, $5m+5-s$, of all points $P$ in that case and that they are parametrized by an irreducible variety (see Remark \ref{e0e}). Each case occurs (Remark \ref{zzz2}).

Case (vii) is obvious. Case (i) is the hard one and we use its proof (called step (b) in the proof of Theorem \ref{i2}) to prove the other cases
and (for the easiest cases) more informations that just the integer $r_{m,d}( P)$. For each $P\in \mathbb {P}^r$ let $\mathcal {S}( P)$ be the set of all sets $S\subset \mathbb {P}^m$
evincing the rank of $P$, i.e. the set of all subset $S\subset \mathbb {P}^m$ such that $P\in \langle \nu _d(S)\rangle$ and $\sharp (S) =r_{m,d}( P)$ (see Proposition \ref{w3}
(case arbitrary cactus rank $b\ge 2$ and type $(b-1;2,1,\dots )$, Proposition \ref{w5.0} (case $s=3$, $b_1=b_2=2$, $b_3=1$), Proposition \ref{w4} (case
arbitrary cactus rank $b\ge 3$ and type $(b-2;3,1,\dots )$
and Proposition \ref{w5} (case $s=2$, $b_1=3$, $b_2=2$)). In each case $P$ depends on exactly $5$ variables, because $P\in \langle \nu _d(A)\rangle$.

\begin{remark}\label{ww1} If $d\ge 9$, then for each $P\in \mathbb {P}^r$ with border rank $5$ there is a unique degree 5 scheme $A$ evincing the cactus rank of $P$ (\cite[Corollary 2.7]{bgl} or \cite[Lemma 4.2 and text after Remark 2.7]{bb2}).
Fix any degree $5$ zero-dimensional scheme $A$ achieving the cactus rank of some $P\in \mathbb {P}^r$ with border rank $5$. Fix any $Q\in \langle \nu _d(A)\rangle$
and let $E\subseteq A$ be a minimal subscheme of $A$ such that $P\in \langle \nu _d(E)\rangle$. If $d\ge 9$ we get that $b_{m,d}(Q) =z_{m,d}(Q) =\deg (E)$ and
that $E$ is the only scheme evincing the cactus rank of $Q$. $Q$ depends essentially on 5 variables if and
only if $E =A$ and $\dim (\langle A\rangle )=\mathbb {P}^4$. In each step of the proof of Theorem \ref{i2} we give a classification of all $A$ appearing as cactus rank
and an easy way to write $A$. Therefore, using $A$, one may easily give $P\in \mathbb {P}^r$ with a prescribed rank, among the possible ones
$4d-3$, $3d-1$, $2d+1$, $d+3$, $5$. Without knowing $A$ it may be possible to detect the value of $r_{m,d}( P)$ for a specific $P$, just
because Theorem \ref{i2} list all possible ranks and the difference between two different ranks is huge for large $d$.
\end{remark} 

\begin{remark}\label{ww2}
Take the set-up of Theorem \ref{i2}, but only assume $d\ge 4$. Fix $P\in \mathbb {P}^r$ with border rank $5$ and assume
that there is no 3-dimensional linear subspace $H\subset \mathbb {P}^m$ such that $P\in \langle \nu _d(H)\rangle$. Since $d\ge 4$, there is at least one  scheme $A$ evincing
the cactus rank of $P$. Call $(s;b_1,\dots ,b_s)$, $b_1+\cdots +b_s=5$, the type of $A$. The classification of the possible connected components $A_1,\dots ,A_s$
of $A$ is done in the corresponding step of the proof of Theorem \ref{i2}, because it does not depend on $d$. At the beginning of each step we prove an upper bound for
the integer $r_{m,d}( P)$ (the hard part was then to check the opposite inequality). This part of the inequality does not depend on $d$. Hence in each case
$r_{m,d}( P)$ is at most the value listed in Theorem \ref{i2}. In particular we get $r_{m,d}( P) \le 4d-3$.\end{remark}

We work over an algebraically closed field $\mathbb {K}$ with $\mbox{char}(\mathbb {K}) =0$. Any easy weak form of Lefschetz's principle would allow to deduce this case from the case $\mathbb {K}=\mathbb {C}$.

\section{Preliminaries}\label{S2}

Let $X$ be a projective variety, $D$ an effective Cartier divisor of $X$ and $Z\subset X$ a zero-dimensional scheme. The residual scheme $\mbox{Res}_D(Z)$ of $Z$ with respect
to $D$ (or with respect to the inclusion $D\subset X$) is the closed subscheme of $X$ with $\mathcal {I}_Z:\mathcal {I}_D$ as its ideal sheaf. For every line bundle $\mathcal {L}$ on $X$ we have a residual exact sequence of coherent sheaves
\begin{equation}\label{eqres}
0 \to \mathcal {I}_{\mbox{Res}_D(Z)}\otimes \mathcal {L}(-D) \to \mathcal {I}_Z\otimes \mathcal {L}\to \mathcal {I}_{Z\cap D}\otimes \mathcal {L}_{|D}\to 0.
\end{equation}

The case $m=2$ of the next lemma is a very particular case of \cite[Corollaire 2]{ep}; the general case easily follows by induction on $m$, taking a hyperplane $H\subset \mathbb {P}^m$ with maximal
$\deg (Z\cap H)$ as in step (b) of the proof of Theorem \ref{i2} (case $g\le 2$, $g'=1$); if $\deg (Z)\le 2d+1$, then the lemma is \cite[Lemma 34]{bgi}.

\begin{lemma}\label{w1}
Fix an integer $d\ge 6$. Let $Z\subset \mathbb {P}^m$, $m\ge 2$, be a zero-dimensional scheme with $\deg (Z)\le 3d+1$ and $h^1(\mathcal {I}_Z(d)) >0$.
Then either there is a line $L\subset \mathbb {P}^m$ with $\deg (L\cap Z)\ge d+1$ or there is a conic $T\subset \mathbb {P}^m$ with $\deg (T\cap Z)\ge 2d+2$
or there is a plane cubic $F$ with $\deg (F\cap Z)\ge 3d$.\end{lemma}

\begin{remark}\label{w2}
Take the set-up of Lemma \ref{w1} and assume the existence of a plane conic $T$
such that $\deg (T\cap Z)\ge 2d+2$, but that there is no line $L\subset \mathbb {P}^m$ with $\deg (L\cap Z)\ge d+2$. When $Z$ has many reduced connected components, it is obvious that $T$ must be reduced. Assume that $T$ is reduced, but reducible, say $T = D\cup R$, with $D$ and $R$ lines. Set $\{o\}:= D\cap R$. Since
$\deg (D\cap Z)\le d+1$ and $\deg (R\cap Z) \le d+1$, we get $\deg (D\cap Z)=\deg (R\cap Z)=d+1$ and that $Z$ is a Cartier
divisor of the nodal curve $T$ (it is a general property of nodal curves with smooth irreducible component: for any zero-dimensional scheme $E\subset D\cup R$
we have $\deg (E\cap D)+\deg (E\cap R)-1\le \deg (E)\le \deg (E\cap D)+\deg (E\cap R)$ and $\deg (E) =\deg (E\cap D) +\deg (E\cap R)$ if and only if
$E$ is a Cartier divisor of $D\cup R$). Now assume that existence of a plane cubic $F$ with $\deg (F\cap Z)\ge 3d$. $F$ is not reduced
if and only if there is a line $L\subset F$ appearing in $F$ with multiplicity at least two. To get that $F$ is reduced it is sufficient to assume that $\deg (R\cap Z)\le d+1$
for each line $R$ and that $Z$ has at least $2d+2$ reduced connected components.
\end{remark}

\begin{remark}\label{w3.0}
Fix integers $m\ge 2$ and $d\ge 3$. Let $Z\subset \mathbb {P}^m$ be a subscheme such that $Z$ spans $\mathbb {P}^m$, $\deg (Z)\le d+1+m$ and $h^1(\mathcal {I}_Z(d)>0$.
Then $\deg (Z) =d+1+m$ and there is a line $L\subset \mathbb {P}^m$ such that $\deg (Z\cap L) = d+2$ (a quick proof: use induction on $m$ and the proof of Theorem \ref{i2} below
using the hyperplanes $H_i$; we get $g\le 2$).
\end{remark}

\begin{remark}\label{uno.1}
Let $T\subset \mathbb {P}^2$ be a reduced curve of degree $t<d$. It is connected and the projective space $\langle \nu _d(T))\rangle$ has dimension
$x:= {{d+2}\choose{2}} -{{d-t+2}\choose{2}}-1$. Every point of $\langle \nu _d(T)\rangle$ has rank at most $x$ with respect to the curve $\nu _d(T)$ (the proof of
\cite[Proposition 5.1]{lt} works verbatim for reduced and connected curves). Hence $\sharp (B\cap T) \le x$. If $t=1$ (resp. $t=2$, resp $t=3$) then $x =d$ (resp. $2d$, resp. $3d-1$).
\end{remark}

\section{Proof of Theorem \ref{i2} and related result}\label{S3}

Step (b) of the proof of Theorem \ref{i2} is by far the most difficult part of this paper and the one which may be used elsewhere. We outline here the proof of Theorem \ref{i2}
and of the other results of this section.

\subsection{Outline of the proof of Theorem \ref{i2} and of the other results of this section} 

We first give the classification of all possible degree $5$ schemes $A$. The key step (called step (b))
is the one in which we handle the case $s=1$ (it is an enhanced version of the proof of \cite[Proposition 5.19]{bb2}). This step is subdivided into several substeps and sub-substeps  (up to, e.g., step (b2.2.2.1), which is the first step of step (b2.2.2),
which is the second step of step (b2.2), which is the second step of step (b2), which is the second step of step (b)). There are several good reasons for the splitting of the proof in this way.
In the proof of the other cases and in the proof of Propositions \ref{w3}, \ref{w5.0}, \ref{w4}, \ref{w5} we need to quote a specific substep or to modify it a little bit; in this way all proofs
are detailed and complete, but with no duplication, and each part of the proof has its own label, making easier to quote it in future works. In step (c) of the proof of Theorem \ref{i2} we reduce some cases to Propositions \ref{w3}, \ref{w5.0}, \ref{w4} and \ref{w5}.
Proposition \ref{w3} covers the case $b_1=2$ and $b_2=b_3=b_4=1$. Proposition \ref{w5.0} covers the case $b_1=b_2 =2$ and $b_3=1$. Proposition \ref{w5} covers the case $b_1=3$ and $b_2=b_3=1$. Proposition \ref{w4} covers the case $b_1=3$ and $b_2 =2$. In step (d) of the proof of Theorem \ref{i2} we prove the case $s=2$, $b_1=4$, $b_2=1$(this is far easier than in step (b)). Proposition \ref{w3} covers all points $P$ with $b_{m,d}( P) =z_{m,d}( P) =b$ evinced by a scheme of type $(b-1;2,\dots ,1)$ (it also describes
all $B\in \mathcal {S}( P)$, i.e. all sets $S\subset \mathbb {P}^m$
such that $P\in \langle \nu _d(S)\rangle$ and $\sharp (S) =r_{m,d}( P)$).
Proposition \ref{w5.0} gives the rank and a description of all $B\in \mathcal {S}( P)$ when $s= 3$, $b_1=b_2=2$ and $b_3=1$. Proposition \ref{w5} describes the case
$s=2$, $b_1=3$, $b_2=2$. The proof of Theorem \ref{i2} may be seen as an iteration (several times) of the proofs in \cite[\S 5.2]{bb2}.

\begin{proof}[Proof of Theorem \ref{i2}]
It is a general result that every zero-dimensional scheme evincing the cactus rank of some $P\in \mathbb {P}^r$ is Gorenstein  (\cite[Lemma 2.3]{bb+}). There is a scheme, $A$, evincing the border rank $b_{m,d}( P)$ of $P$, because $d\ge b_{m,d}( P)-1$ (\cite[Lemma 2.6]{bgl}). This scheme is unique, because $d\ge 9 =2b_{m,d}( P)-1$ (\cite[Corollary 2.7 and its proof]{bgl} or \cite[Lemma 4.2 and the text after Remark 2.7]{bb2}).
It also evinces the cactus rank of $P$, i.e. $z_{m,d}( P) =5$, because \cite{bgi} and \cite{bb2} give the classifications for all points with cactus rank $\le 4$ and they have
border rank $\le 4$. 
Since $P\notin \langle \nu _d(H)\rangle$ for any 3-dimensional linear subspace $H\subset \mathbb {P}^m$, $\deg (A) =5$ and $P\in \langle \nu _d(A)\rangle$, we have
$\dim (\langle A\rangle ) =4$, i.e. $A$ is linearly independent. In particular it is in linearly general position in $\langle A\rangle$. The scheme $A$ is not
the double $2O$ of one point of $\langle A\rangle$ (i.e. the closed subscheme of the projective space $\langle A\rangle$ with $(\mathcal {I} _{O,\langle A\rangle})^2$ as its ideal sheaf), because the proof of \cite[Theorem 32]{bgi} shows that for each $P\in \langle \nu _d(2O)\rangle$ there
is a line $L\subset \mathbb {P}^m$ such that $P\in \langle \nu _d(L)\rangle$; alternatively, use that $A$ is Gorenstein by \cite[Lemma 2.3]{bb+}. By \cite[Theorem 1.3]{eh} $A$ is curvilinear.

By assumption $A$ is linearly independent and it spans a 4-dimensional linear subspace of $\mathbb {P}^m$. By concision (\cite[Exercise 3.2.2.2]{l}),
it is sufficient to do the case $m=4$. Write $A = A_1\sqcup \cdots \sqcup A_s$ with $\deg (A_i) =b_i$, $b_i\ge b_j$ if $i\le j$, and set $\{O_i\}:= (A_1)_{red}$. Since $A$ evinces the cactus rank of $P$, then $P\notin \langle \nu _d(A')\rangle$
for any $A'\subsetneq A$. Let $B\subset \mathbb {P}^m$
any set evincing the rank of $P$. Set $W_0:= A\cup B$. If $A\ne B$, then $h^1(\mathcal {I}_{W_0}(d)) >0$ (\cite[Lemma 1]{bb1}). If $s\ne 5$, then $A\ne B$.

\quad (a) If $A$ is reduced, i.e. if $s=5$, then $r_{4,d}( P) \le 5$. Since we assumed border rank 5, we get $r_{4,d}( P) =5$. 

\quad (b) Now assume $s=1$ and hence $b_1=5$. The scheme $A_1$ is curvilinear and unramified, because $\mathbb {P}^4 =\mathbb {P}(W)$ with
$\dim (W) =5$, $\langle A_1\rangle =\mathbb {P}^4$ and we may apply \cite[Theorem 1.3]{eh}. 

\quad {\emph {Claim:}} $A_1$ is contained in a rational normal curve $C\subset \mathbb {P}^4$.

\quad {\emph {Proof of the Claim:}} Since $A_1$ is connected and curvilinear for each $i=1,\dots ,5$, $A_1$ has a unique subscheme $A_1[i]$ of $A_1$ such
that $\deg (A_1[i]) =i$. Since $A_1$ is linearly independent, then $\dim (\langle A_1[i]\rangle )=i-1$. For a parametrization $t\mapsto (t,t^2,t^3,t^4)$ of $C$ take the flag of linear subspaces $\langle A_1[1]\rangle \subset \langle A_1[2]\rangle \subset \cdots \subset  \langle A_1[5]\rangle$.

The existence of the rational normal curve $C$ also gives that $A_1$ is unique, up to a projective transformation. The curve $\nu _d( C)$ is a degree $4d$ rational normal curve in its linear span.
By Sylvester's theorem (\cite{cs}, \cite[Theorem 4.1]{lt}, \cite[Theorem 23]{bgi}) the rank of $P$ with respect to $\nu _d( C)$ is $4d-3$. Hence $r _{4,d}( P) \le 4d-3$. Assume $\sharp (B) \le 4d-4$ and hence $\deg (W_0) \le 4d+1$. Let $H_1$ be a hyperplane such that $e_1:= \deg (W_0\cap H_1)$ is maximal. Set $W_1:= \mbox{Res}_{H_1}(W_0)$. Fix an integer
$i\ge 2$ and assume to have defined the integers $e_j$, the hyperplanes $H_j$ and the scheme $W_j$, $1\le j <i$. Let $H_i$ be any hyperplane such that
$e_i:= \deg (H_i\cap W_{i-1})$ is maximal. Set $W_i:= \mbox{Res}_{H_i}(W_{i-1})$. We have $e_i\ge e_{i+1}$ for all $i$. For each integer $i>0$ we have the residual exact sequence
\begin{equation}\label{eqc1}
0 \to \mathcal {I}_{W_i}(d-i) \to \mathcal {I}_{W_{i-1}}(d+1-i) \to \mathcal {I}_{W_{i-1}\cap H_i,H_i}(d+1-i) \to 0
\end{equation}
Since $h^1(\mathcal {I}_{W_0}(d)) >0$, there is an integer $i>0$ such that $h^1(H_i,\mathcal {I}_{W_{i-1}\cap H_i,H_i}(d+1-i)) >0$. We call $g$ the minimum such an integer.
Since any zero-dimensional scheme
with degree $4$ of $\mathbb {P}^4$ is contained in a hyperplane, if $e_i\le 3$, then $W_i= \emptyset$ and $e_j=0$ if $j>i$. Hence $e_{d+2}=0$ and $e_{d+1}\le 1$.
Since $h^1(\mathcal {O}_{\mathbb {P}^4}(d))=h^1(\mathcal {I}_Q) =0$ for any $Q\in \mathbb {P}^4$, we get $g\le d$. By \cite[Lemma 34]{bgi} either $e_g \ge 2(d+1-g)+2$ or there is a line $L\subset H_g$
such that $\deg (L\cap W_{i-1}) \ge d+3-g$. Assume for the moment $g\ge 2$ and $e_g\le 2(d+1-g)+1$. Since $e_g>0$, $W_{g-2}$ spans $\mathbb {P}^4$.
Hence $e_{g-1} \ge d+5-g$. This inequality is obvious if $e_g\ge 2(d+1-g)+2$ and $g\le d-1$, because $2(d+1-g)+2 \ge d+5-g$ if $g\le d-1$. Now assume
$g=d$; if $e_d \le 4$ and there is no line $L\subset H_d$ with $\deg (L\cap W_{d-1})\ge 3$, then $e_d=4$ and $W_{d-1}\cap H_d$ is contained in a plane; hence even in this case we get $e_{g-1}\ge 5 = d+5-g$. Hence if $g\ge 2$, then $e_{g-1} \ge d+5-g$. Since $e_i\ge e_{i+1}$ for all $i$, we
get that $4d+1 \ge \deg (W_0)\ge g(d+5-g) -2$. Set $\psi (t) = t(d+5-t) -2$. The function $\psi : \mathbb {R} \to \mathbb {R}$ is increasing if $t\le (d+5)/2$, decreasing if $t\ge (d+5)/2$ and strictly monotone
in any interval not containing $(d+5)/2$. Since $\psi (4) = 4d+2 > 4d+1$ and $\psi (d) =5d-2>4d+1$, we get $1\le g \le 3$.

\quad (b1) Assume $g=3$. We have $e_3 \ge d$ and $e_1\ge e_2 \ge d+2$. Hence $\deg (W_3) \le d-3$ and $e_1\le 2d-1$. Since $h^1(\mathcal {I}_{W_4}(d-3)) =0$, \cite[Lemma 5.1]{bb2}
gives $W_0\subset H_1\cup H_2\cup H_3$. Since $e_1\ge e_2\ge e_3$ and $e_1+e_2+e_3\le 4d+1$, we have $e_3 \le (4d+1)/3 \le 2(d-2)+1$. By \cite[Lemma 34]{bgi}
there is a line $L\subset H_3$ such that $\deg (L\cap W_3) \ge d$. Taking instead of $H_3$ any hyperplane $M\subset L$ we get $W_0\subset H_1\cup H_2\cup M$.
In particular we have $B\subset H_1\cup H_2\cup L$. Since $A$ is curvilinear, we even get $W_0\subset H_1\cup H_2\cup L$. Set $Z_0:= W_0$. Let $N_1\subset \mathbb {P}^4$ be a hyperplane containing $L$ and such that
$f_1:= \deg (W_0\cap N_1)$ is maximal among the hyperplanes containing $L$. We have $f_1\ge e_3+2 \ge d+2$. Set $Z_1:= \mbox{Res}_{N_1}(Z_0)$. Let $N_2\subset \mathbb {P}^4$ be a hyperplane such that $f_2:= \deg (N_2\cap Z_1)$ is maximal. Set $Z_2:= \mbox{Res}_{N_2}(Z_1)$. For any integer $i\ge 3$ let $N_i\subset \mathbb {P}^4$
be a hyperplane such that $f_i:= \deg (Z_{i-1}\cap N_i)$ is maximal. Set $Z_i:= \mbox{Res}_{N_i}(Z_{i-1})$.
 Let $g'$ be the minimal positive integer $i$ such
that $h^1(N_i,\mathcal {I}_{N_i\cap Z_{i-1},N_i}(d+1-i)) >0$ (it exists by the residual exact sequence (\ref{eqc1}) obtained using the hyperplanes $N_i$). Now we only have that $f_i\ge f_{i+1}$ if $i\ge 2$, but we also have $\deg (Z_0)-f_1 \le 3d-1$. Hence $g'\le 3$. 

\quad (b1.1) Assume $g'=3$. We have $f_3 \le (3d-1)/2 \le 2(d-2)+1$. Hence there is a line $R\subset N_3$ such
that $\deg (R\cap Z_2) \ge d$. Since $B\cap L\cap W_1=\emptyset$, we have $R\ne L$. Set $Z'_0:= Z_0$. Take any hyperplane $N'_1$ containing $L\cup R$ and with
$f'_i:= \deg (N'_1\cap Z_0)$ maximal among the hyperplanes containing $L\cup R$. Set $Z'_1:= \mbox{Res}_{N'_1}(Z'_0)$. Define as above the integers $f'_i$, the hyperplanes $N'_i$ and the schemes $W'_i$ with $f'_i\ge f'_{i+1}$ for all $i\ge 2$ and $f'_1 \ge 2d-1$. Let $g''$ be the minimal positive integer $i$ such
that $h^1(N'_i,\mathcal {I}_{N'_i\cap Z'_{i-1},N'_i}(d+1-i)) >0$ (it exists by the residual exact sequence (\ref{eqc1}) obtained using the hyperplanes $N'_i$).
Since $f'_1\ge 2d-1$, we have $g'' \le 2$.

\quad (b1.1.1) Assume for the moment $g'' =2$. We get $f'_2\ge d+1$ and the existence of a line $D\subset N'_2$ such that $\deg (D\cap Z'_2) \ge d+1$ (\cite[Lemma 34]{bgi}).

First assume that $L\cup D\cup R$ is contained in a hyperplane $M$. In this case we get $e_1 \ge \deg (M\cap W_0) \ge 3d-1$. Hence $e_3+e_2\le d+2$, a contradiction.
Now assume
that $L\cup D\cup R$ is not contained in a hyperplane. Since $h^0(\mathcal {O}_{\mathbb {P}^4}(2)) =15$ and
$h^0(\mathcal {O}_{L\cup R\cup D}(2)) \le 9$, we have $h^0(\mathcal {I}_{L\cup R\cup D}(2)) \ge 6$. Since $L\cup D\cup R$ is not contained in a hyperplane, either these lines are pairwise disjoint
and they are not contained in a plane
or two of them meets and the other one is disjoint from the plane spanned by the first two lines. Therefore $L\cup R\cup D$ is the scheme-theoretic base locus
of $|\mathcal {I}_{L\cup R\cup D}(2)|$. Let
$Q$ be a general hyperquadric containing $L\cup D\cup R$.  A dimensional count gives that $Q$ is irreducible.
First assume $h^1(\mathcal {I}_{\mbox{Res}_Q(W_0)}(d-2)) =0$. By \cite[Lemma 5.1]{bb2} we get $A\cup B\subset Q$. Since $L\cup R\cup D$ is the scheme-theoretic base locus
of $|\mathcal {I}_{L\cup R\cup D}(2)|$ and $A$ is curvilinear, we get $A\subset L\cup R\cup D$, contradicting the fact that at most two of these lines contain $O$.

Now assume $h^1(\mathcal {I}_{\mbox{Res}_Q(W_0)}(d-2))>0$. Since $\deg (\mbox{Res}_Q(W_0)) \le d+2$, we get the existence of a line $R'$ such $
\deg (R' \cap \mbox{Res}_Q(W_0)) \ge d$. We have $h^0(\mathcal {I}_{L\cup D\cup R\cup R'}(2)) \ge 3$ and we take a general $Q_1\in |\mathcal {I}_{L\cup R'\cup R\cup D}(2)|$. As above we find $A\cup B\subset Q_1$. Since $e_1\le 2d-1$, no 3 of the lines $L$, $R$, $R'$ and $D$ are contained in a plane.
Hence at most two of these lines contain $O$ and at most one of the lines is tangent to $C$ at $O$. We get $\deg (B) \ge -3 +\deg (W_0\cap L)+\deg (W_0\cap R)+\deg (W_0\cap R')
+\deg (W_0\cap D) \ge -3+d+d+d+1+d =4d-2$, a contradiction.

\quad (b1.1.2) Now assume $g''=1$. Since $A$ is connected and not contained in a hyperplane, \cite[Lemma 5.1]{bb2} gives $h^1(\mathcal {I}_{Z'_1}(d-1)) >0$.
We have $\deg (Z'_1) \le 3d-1$. By Lemma \ref{w1} either there is a line $R\subset \mathbb {P}^4$ such that $\deg (R\cap Z'_1) \ge d+1$ or there is a reduced
conic $T\subset \mathbb {P}^4$ such that $\deg (T\cap Z'_1) \ge 2d$. In the latter case we would have $e_1\ge 2d+1$, a contradiction. Hence there
is a line $R\subset \mathbb {P}^4$ such that $\deg (R\cap Z'_1) \ge d+1$. Since $R\cap (B\setminus R\cap L) \supseteq R\cap (B\setminus B\cap N'_1)\ne \emptyset$, then $R\ne L$ and hence at most one of the lines $R, L$
is tangent to $C$ at $O$. If $O\notin (L\cap R)$, then $\deg (W_0\cap (L\cup R)) \ge 2d$ and hence $e_1\ge 2d$, a contradiction.
Now assume $O\in (L\cap R)$. At most one of the lines $L, R$ is tangent to $C$ at $O$. We get $\deg (W_0\cap (L\cup R)) \ge 2d-1$.
By concision (\cite[Exercise 3.2.2.2]{l}) $B$ spans $\mathbb {P}^4$. Hence there is a hyperplane $M\supset L\cup R$ containing a point
of $B\setminus B\cap (L\cup R)$. Hence $e_1\ge 2d$, a contradiction.

\quad (b1.2) Assume $g' =2$. Since $f_2 \le e_1\le 2d-1$, there
is a line $R\subset N_2$ such that $\deg (R\cap Z_1)\ge d+1$. Since $R\cap (B\setminus B\cap L) \subseteq R\cap B\cap Z_1$ and $\deg (R\cap B\cap Z_1) \le
\deg (R\cap (B\setminus B\cap L)) +\deg (R\cap A) \le \deg (R\cap (B\setminus B\cap L)) +2$, we get $R\ne L$ and hence $e_1\ge -1 +\deg (L\cap W_0)+\deg (R\cap W_0)
\ge 2d$, a contradiction.

\quad (b2) Assume $g=2$. Since $e_2\le e_1$ and $e_1+e_2\le 4d+1$, we have $e_2\le 2d$. By Lemma \ref{w1} either there is a line $L\subset H_2$ such
that $\deg (L\cap W_1)\ge d+1$ or $e_2=2d$ and there is a conic $T\subset H_2$ such that $W_1\cap H_2 \subset T$, but there is no line $L\subset T$
with $\deg (L\cap W_0)\ge d+1$; in the latter case $T$ is reduced and if $T$ is reducible, then $\deg (J\cap W_1) =d$ for every component of $T$ (Remark \ref{w2}).

\quad (b2.1) Assume $e_2=2d$ and the existence of a conic $T\subset H_2$ such that $W_1\cap H_2 \subset T$. Since $A$ spans $\mathbb {P}^4$, then $A\nsubseteq T$
and hence $O\in T$. Since $B$ spans $\mathbb {P}^4$ (\cite[Exercise 3.2.2.2]{l}),
there is $o\in B\setminus B\cap \langle T\rangle$. Let $M$ be the plane spanned by $o$ and $T$. Since $e_1\le 2d+1$, we
get $e_1=2d+1$, $\sharp (B) =4d-4$, $O\notin B$ and that $W_0\cap M =\{o\} \sqcup (W_1\cap T)$ with $\deg (W_1\cap T) = 2d$. Therefore $B = (B\cap H_1)\sqcup (B\cap T)$
and no point of $B\cap H_1$ is contained in the plane $\langle T\rangle$. Since $A\nsubseteq M$, \cite[Lemma 5.1]{bb2}
gives  $h^1(\mathcal {I}_{\mbox{Res}_M(W_0)}(d-1)) >0$. Lemma \ref{w1} and Remark \ref{w2}
give that either there is a line $R$ with $\deg (R\cap \mbox{Res}_M(W_0)) \ge d+1$ or $\deg (\mbox{Res}_M(W_0)) =2d$ and $\mbox{Res}_M(W_0)$ is contained in a reduced conic $T'$. The plane $M$ and hence
$R$ or $T$ depends on the choice of $o$ (call them $M_o,R_o,T_o$).

\quad (b2.1.1) First assume that $\deg (\mbox{Res}_M(W_0)) =2d$ and that $\mbox{Res}_M(W_0)$ is contained in a reduced conic $T'$, but there is no line $R$ with
$\deg (R\cap \mbox{Res}_M(W_0)) \ge d+1$; hence if $T'$ is reducible, then $\deg (J\cap \mbox{Res}_M(W_0))=d$ for each component of $T'$ and the
singular point of $T'$ is not contained in $\mbox{Res}_M(W_0)$. Fix $q\in T'\cap B$ and take the hyperplane
$M_q:= \langle \{q\}\cup T\rangle$. We have $\mbox{Res}_{M_q}(B) = ((T'\cap B)\setminus \{q\}) \sqcup \{o\}$. This set is not contained in a line (even if $T'$ is reducible), because $d\ge 9$. This set is contained in a conic only if $o\in T'$. Hence (applying \cite[Lemma 5.1]{bb2} to $M_q$) we get $B\subset T\cup T'$. As in step (b2.1) we have
$O\in T'$. Hence $O\in T\cap T'$. If one of the two conics is reducible, then taking a connected union $J\subset T\cup T'$ of 3 lines we get $e_1 \ge 2d+2$, a contradiction. Hence both conics $T, T'$ are smooth.
Since $\langle B\rangle =\mathbb {P}^4$, we have $\langle T\rangle \cap \langle T'\rangle = \{O\}$. Hence at most one of the two conics is tangent to $C$ at $O$.
If $T'$ is not tangent to $C$ at $O$ we get $\sharp (B\cap T)\ge 2d-2$ and $\sharp (B\cap T')\ge 2d-1$, a contradiction (and similarly if $T$ is not tangent to $C$ at $O$).

 \quad (b2.1.2) Now assume the existence of the line $R$. Since $e_1=2d+1$, then $T\cup R$ spans $\mathbb {P}^4$, i.e. $R\cap \langle T\rangle = \emptyset$.
 In particular $O\notin R\cap T$. First assume $O\notin R$. We have $\sharp (B\cap R) =d+1$, contradicting the case $t=1$ of Remark \ref{uno.1}.
 Now assume $O\notin T$ and hence $\sharp (B\cap T) =2d$. Since $O\in R$, we get $O\notin \langle R\rangle$ and either $\sharp (B\cap (R\setminus \{O\})) = d$
 or $R$ is the tangent line to $C$ at $O$. Since $h^0(\mathcal {I}_{T\cup R}(2)) =7$, there is $Q\in |\mathcal {I}_{T\cup R}(2)|$ with $\deg (Q\cap W_0) \ge 3d+2$ (e.g., if $R$ is not tangent to $C$ at $O$ it is sufficient to take $Q$ containing $R\cup T$ and the degree two zero-dimensional subscheme $\eta$ of $A$). Since $\deg (Q\cap \mbox{Res}_Q(W_0))\le d-1$,
 we have $h^1(\mathcal {I}_{\mbox{Res}_Q(W_0)}(d-2)) =0$ and hence $Q\supset W_0$ (\cite[Lemma 5.1]{bb2}), while we may even find $Q\in |\mathcal {I}_{T\cup R\cup \eta}(2)|$ not containing
 the degree $4$ subscheme of $A$, because $A$ is curvilinear.

\quad (b2.2) Assume the existence of a line $L\subset H_2$ such that $\deg (L\cap W_1) \ge d+1$. The case $t=1$ of Remark \ref{uno.1} gives $O\in L$,
$d-1 \le \sharp (B\cap L) \le d$ and that $\sharp (B\cap L)=d$ and $O\notin B$ if $L$ is not the tangent line to $C$ at $O$. Let $N_1\subset \mathbb {P}^4$
be a hyperplane containing $L$ and such that $f_1:= \deg (W_0\cap L)$ is maximal among the hyperplanes containing $L$. Since $A$ spans $\mathbb {P}^4$, we
have $f_1\ge d+3$ and hence $e_1\ge d+3$. Set $Z_0:= W_0$ and $Z_1:= \mbox{Res}_{N_1}(Z_0)$. Let $N_2$ be a hyperplane such that $f_2:= \deg (N_2\cap Z_1)$. Set
$Z_1:= \mbox{Res}_{N_2}(Z_1)$. Define inductively the hyperplane $N_i$, $i\ge 3$, the schemes $Z_i:= \mbox{Res}_{N_i}(Z_{i-1})$ and
the integer $f_i:= \deg (N_i\cap Z_{i-1}) $ with $f_i$ maximal among all hyperplanes. We have $f_i\ge f_{i+1}$ for all $i\ge 3$ and if $f_i\le 3$, then $f_{i+1}=0$
and $Z_i=\emptyset$. We have the residual exact sequence similar to (\ref{eqc1}) with $N_i$ instead of $H_i$, $Z_{i-1}$ instead of $W_{i-1}$ and $Z_i$ instead of $W_i$.
Hence there is an integer $i>0$ such that $h^1(N_i,\mathcal {I}_{N_i\cap Z_{i-1},H_i}(d+1-i)) >0$ and we call $g'$ the first such an integer.
Since $\sum _{i\ge 2} f_i \le 3d-2$ and $f_{i+1}=0$ if $f_i\le 3$, we have $g'\le d+1$. Assume for the moment $g'\le d+1$. Hence either $f_{g'} \ge 2(d+1-g')+2$ or there is a line $R\subset N_{g'}$
such that $\deg (R\cap Z_{i-1})\ge d+3-g'$. Assume for the moment $3 \le g' \le d$. We get $f_{g'-1} \ge d+5-g'$ (\cite[Lemma 34]{bgi}). Hence $3d-2 \ge \sum _{i\ge 2} f_i \ge (g'-1)(d+5-g') -1$. As in step
(b) we also exclude the case $g'=d$.
We get $g' \le 3$ if $g\le d$.

Now assume $g' =d+1$, i.e. assume $\deg (W_2)\ge 2$. Since $f_1\ge d+3$ and $f_i\ge 4$ for $2\le i \le d$, we get $4d+1 \ge d+3 +4(d-1) +2$, a contradiction.

\quad (b2.2.1) Assume $g'=3$. Since $f_1\ge d+3$, $f_2\ge f_3$ and $f_1+f_2+f_3 \le 4d+1$, we get $f_3\le (3d-2)/2 \ge 2(d-2)+1$ (since $d\ge 4$).
By \cite[Lemma 34]{bgi} there is a line $R\subset N_3$ such that $\deg (R\cap Z_2)\ge d$. Since $f_3>0$, we also see that $Z_1$ spans $\mathbb {P}^4$.
Hence $f_2\ge d+2$. Therefore $f_1 \le 2d-1$. Let $M_1$ be a hyperplane containing $L\cup R$ and with $h_1:= \deg (M_1\cap 0)$
maximal. We have $f_1\ge h_1$. If $L\cap R =\emptyset$, then $h_1\ge \deg (L\cap W_0)+\deg (R\cap W_0) \ge 2d+1$, a contradiction. Now
assume $L\cap R \ne \emptyset$. We have $R\ne L$, because $R\cap (B\setminus B\cap L)\ne \emptyset$. Since $L\cap R$ is scheme-theoretically a point,
we have $\deg (W_0\cap (L\cup R)) \ge \deg (W_0\cap L)+\deg (W_0\cap R)-1$. Hence
$h_1\ge 1+ \deg (W_0\cap L)+\deg (W_0\cap R)-1\ge 2d$, a contradiction.

\quad (b2.2.2) Assume $g'=2$. Since $e_2\ge d+1$, we have $e_1\le 3d$.
Since $f_1\ge d+3$, we have $f_2\le 3d-2 = 3(d-1)+1$. By Lemma \ref{w1} and Remark \ref{w2} either
there is a line $R\subset N_2$ with $\deg (R\cap Z_1)\ge d+1$ or there is a reduced conic $T\subset N_2$ with $\deg (Z_2\cap T) \ge 2d$
or there is reduced plane cubic $F\subset N_2$ with $\deg (F\cap N_2) \ge 3d-3$.

\quad (b2.2.2.1) Assume the existence of a reduced plane cubic $F\subset N_2$ with $\deg (F\cap N_2) \ge 3d-3$ and take any hyperplane $M$ containing
$F$ and another point of $B$ (it exists, because $\langle B\rangle =\mathbb {P}^4$ by \cite[Exercise 3.2.2.2]{l}). By \cite[Lemma 5.1]{bb2} we have $h^1(\mathcal {I}_{\mbox{Res}_M(W_0)}(d-1)) >0$. Since $\deg (\mbox{Res}_M(W_0)) \le d+4 \le 2(d-1)+1$ ($d\ge 5$),
there is a line $R\subset \mathbb {P}^4$ such that $\deg (R\cap \mbox{Res}_M(W_0))\ge d+1$. The case $t=1$ of Remark \ref{uno.1} implies
$O\in R$. Since $\deg (R\cap A)\le 2$, we have $\sharp (B\setminus B\cap M) \ge d-1$ with equality
only if $R$ is the tangent line of $C$ at $O$. Assume for the moment $O\notin F$. We get
$\sharp (B\cap M) \ge 1+ \sharp (B\cap F)\ge 3d-2$ and hence $\sharp (N) \ge 3d-3$, a contradiction. Now assume $O\in F$. Since $\langle F\cup R\rangle \ne \mathbb {P}^4$,
we get $e_1 \ge (3d-3)+d-1$, a contradiction.

\quad (b2.2.2.2) Assume the existence of a reduced conic $T\subset N_2$ with $\deg (Z_2\cap T) \ge 2d$; if $T$ is reducible also assume that $\deg (J\cap Z_2)=d$
for every component $J$ of $T$ (Remark \ref{w1}). First
assume that $L$ is an irreducible component of $T$; since $L\subset N_1$, we have $\deg (L\cap Z_1)\le \deg (L\cap A)\le 2$, a contradiction. Now assume that $L$ is not an irreducible component of $T$, but that $L\subset \langle T\rangle$; this case is done as in step (b2.2.2.1). Now assume $L\cap \langle T\rangle =\emptyset$ and hence $O\notin \langle T\rangle$. Let $Q$ be a quadric hypersurface containing $T\cup L$ and at least one other
point of $B$ (it exists, because $h^0(\mathcal {I}_{L\cup T}(2)) = 7$. Since $\deg (W_0)-\deg (W_0\cap Q) \le d-2$, then $h^1(\mathcal {I}_{\mbox{Res}_Q(W_0)}(d-2)) =0$ and hence \cite[Lemma 5.1]{bb2} gives $W_0\subset Q$.
Hence $h^0(\mathcal {I}_{A\cup L\cup T}(2)) \ge 6$. Since $O\notin \langle T\rangle$, we have $h^0(\mathcal {I}_{A\cup T}(2)) =5$, a contradiction.

Now assume $L\cap \langle T\rangle \ne \emptyset$ and
call $N$ any hyperplane containing $T\cup L$. Either by a residual exact sequence (case $h^1(N,\mathcal {I}_{W_0\cap N}(d)) =0$)
or quoting \cite[Lemma 5.1]{bb2} we get $h^1(\mathcal {I}_{\mbox{Res}_N(W_0)}(d-1)) >0$. Since $\deg (\mbox{Res}_N(W_0)) \le d+1$, there is a line
$R\subset \mathbb {P}^4$ such that $\deg (R\cap \mbox{Res}_N(W_0))\ge d+1$. We have $L\ne R$. Remark \ref{uno.1} gives $O\in R$ and 
$\sharp (R\cap (B\setminus N)) \ge d-1$ with strict inequality, unless $\sharp (B\cap L) =d$. Let $Q$ be any quadric surface containing the two conics $L\cup R$ and
$T$ (it exists, because $h^0(\mathcal {I}_{L\cup R\cup T}(2)) \ge 5$). Since $\deg (\mbox{Res} _Q(W_0)) \le 5$, we have $h^1(\mathcal {I}_{\mbox{Res} _Q(W_0)}(d-2))=0$
and hence \cite[Lemma 5.1]{bb2} gives $W_0\subset Q$. Since $A$ is curvilinear, we get $A\cup B\subset L\cup R\cup T$. Since $A\subset B\cup L\cup R\cup T$, the reduced
conic must be reducible and the four lines of $L\cup R \cup T$ pass through the point $O_1$. Moreover, the union of any 3 of these lines span a hyperplane. 
 Write $L\cup R \cup T = L_1\cup L_2\cup L_3\cup L_4$ with each $L_i$ a line.
We have $\sharp (B\cap (L_i\setminus \{O_1\}))+1 \le \deg (L_i\deg (L_i\cap W_0) \le \sharp (B\cap (L_i\setminus \{O_1\}))+2$ and the second inequality is not strict
if and only if $L_i$ is the tangent line of $C$. Hence the second inequality is not strict for at most one index $i$. If $T=L_i\cup L_j$ we also have
$\deg (T\cap W_0) \le \deg (T\cap L_i)+\deg (L_j\cap W_0)$. We get $\sharp (B\setminus \{O_1\}) \ge -5+(d+1)+(d+1)+2d =4d-3$, a contradiction.

\quad (b2.2.3) Assume $g'=1$.

\quad (b2.2.3.1) Assume $\deg (L\cap W_0)\ge d+2$. By the case $t=1$ of Remark \ref{uno.1} we have $\sharp (B\cap L) = d$, $O\notin B$
and $L$ is the tangent line to $C$ at $O$. The maximality property of $N_1$ implies $f_1\ge d+4$ and hence $\deg (Z_1) \le 3d-3$. Since $W_0\nsubseteq N_1$,
then $h^1(\mathcal {I}_{Z_1}(d-1)) >0$. By Lemma \ref{w1} and Remark \ref{w2} either
there is a line $R\subset \mathbb {P}^4$ with $\deg (R\cap Z_1)\ge d+1$ or there is a reduced conic $T\subset \mathbb {P}^4$ with $\deg (Z_1\cap T) \ge 2d$
(if $T$ is reducible we may also assume $\deg (J\cap W_0)=d$ for each irreducible component $J$ of $T$)
or there is reduced plane cubic $F\subset \mathbb {P}^4$ with $\deg (F\cap Z_1) \ge 3d-3$. 

\quad (b2.2.3.1.1) Assume the existence of $F$. In particular we have $e_1\ge 3d-2$. Fix any plane $M\supset L$. Since $\deg ((H_1\cap M)\cap W_0) \ge
4d-1$, we have $h^1(\mathcal {I}_{\mbox{Res}_{H_1\cup M}(W_0)}(d-2)) =0$ and hence $W_0\subset H_1\cup N$. Recall that in this step
$L$ is the tangent line to $C$ at $O$. Set $c:= \deg (A\cap H_1)$. Since $L\notin H_1$, we have $c\le 1$. Since $A\subset H_1\cup N$
for each hyperplane containing $L$, we get $c+2\ge 5$, a contradiction.

\quad (b2.2.3.1.2) Assume the existence of $T$. Since $L\subset N_1$ and $\deg (J\cap Z_1) >\deg (J\cap A)$ for any line $J\subset T$, the set $T\cap L$
is finite. Assume for the moment $L\subset \langle T\rangle$. Since $\deg (W_0\cap \langle T\rangle ) \ge d +\deg (T\cap Z_1) \ge 3d$,
we get $e_1\ge 3d+1$ and hence $e_2\le d$, a contradiction. Now assume $L\nsubseteq \langle T\rangle$ and $L\cap \langle T\rangle \ne \emptyset$.
Let $M$ be the hyperplane spanned by $L\cup T$. Since $\langle T\rangle$ does not contain the tangent line, $L$, of $C$ at $O$, we have
$\deg (A\cap T) \le 1$. Hence $\sharp (B\cap (M\setminus \{O\})) \ge 3d-1$. Hence $e_1\ge 3d$, a contradiction.

Now assume $L\cap \langle T\rangle =\emptyset$. We have $\deg (W_0\cap (L\cup T)) \ge 3d+2$. Since $O\in L$, we have $T\cap A=\emptyset$
and hence $A\nsubseteq T\cup L$. We have $h^0(\mathcal {I}_{L\cup T}(2)) = 7$ and $T\cup L$ is the scheme-theoretic base
locus of $|\mathcal {I}_{L\cup T}(2)|$. Fix any $Q\in |\mathcal {I}_{L\cup T}(2)|$. Since $\deg (\mbox{Res}_Q(W_0)) \le d-1$, we have
$h^1(\mathcal {I}_{\mbox{Res}_Q(W_0)}(d-2)) =0$ and hence $W_0\subset Q$. Moving $Q$ in $|\mathcal {I}_{L\cup T}(2)|$ we get $W_0\subset L\cup T$, a contradiction.

\quad (b2.2.3.1.3) Assume the existence of $R$. Since $R\ne L$ and $L$ is the tangent line to $C$ at $O$, Remark \ref{uno.1} gives $O\in R$ and $\sharp (B\cap (R\setminus \{O\}) =d$. 
Let $M\subset \mathbb {P}^4$ be a hyperplane containing $R\cup L$ and with maximal $\deg (M\cap W_0)$. Since $A\nsubseteq M$, we have $h^1(\mathcal {I}_{\mbox{Res}_M(W_0)}(d-1)) >0$ (\cite[Lemma 5.1]{bb2}). Since $h \ge 1+\deg (W_0\cap \langle R\cup L\rangle )\ge
2d+3$, we have $\deg (\mbox{Res}_M(W_0)) \le 2d-2$. Hence there is a line $D\subset \mathbb {P}^4$ such that $\deg (D\cap \mbox{Res}_M(W_0))\ge d+1$.
Since $L\cup R\subset M$, then $D\cap (B\setminus B\cap M)\ne \emptyset$ and in particular $D\ne L$ and $D\ne L$. Since $L$ is the tangent line to $C$ at $O$, Remark \ref{uno.1} gives $O\in D$ and $\sharp (B\cap (D\setminus \{O\})) =d$. Let $N$ be a hyperplane containing $D\cup R\cup L$. We have $\deg (N\cap W_0) \ge 3d+2$
and hence $h^1(\mathcal {I}_{\mbox{Res}_N(W_0)}(d-1)) =0$. Hence $A\subset N$, a contradiction.

\quad (b2.2.3.2) Assume $\deg (L\cap W_0) =d+1$. By \cite[Lemma 34]{bgi} we have $f_1\ge 2d+2$. Since $W_0\nsubseteq N_1$, we have $h^1(\mathcal {I}_{Z_1}(d-1)) >0$
(\cite[Lemma 5.1]{bb2}). Since $\deg (Z_1) \le 4d+1-f_1\le 2(d-1)+1$, there is a line $R\subset \mathbb {P}^4$ such that $\deg (R\cap Z_1) \ge d+1$. The case
$t=1$ of Remark \ref{uno.1} gives $O\in R$. Since $L\subset N_1$ and $d+1\ge 3$, we have $R\cap (B\setminus B\cap L) \ne \emptyset$
and hence $L\ne R$. Let $M$ be
the hyperplane containing $L\cup R$ and with maximal $g_1:= \deg (M\cap W_0)$. Since $\langle R\cup L\rangle$ is a plane, we have $g_1\ge 1+
\deg (W_0\cap \langle L\cup R\rangle ) \ge 2d$. Since $W_0\nsubseteq M$, we get $h^1(\mathcal {I}_{\mbox{Res}_M(W_0)}(d-1)) >0$.
Since $\deg (\mbox{Res}_M(W_0)) \le 2d+1$, either there is a line $D\subset \mathbb {P}^4$ with $\deg (D\cap \mbox{Res}_M(W_0))\ge d+1$ or there is a conic
$T\subset \mathbb {P}^4$ with $\deg (T\cap \mbox{Res}_M(W_0))\ge 2d$.

\quad (b2.2.3.2.1) Assume the existence of $T$. First, assume $L\subset \langle T\rangle$. Since $\deg (W_0\cap \langle T\rangle ) \ge d +\deg (T\cap Z_1) \ge 3d$,
we get $e_1\ge 3d+1$ and hence $e_2\le d$, a contradiction. Now, assume $L\nsubseteq \langle T\rangle$ and $L\cap \langle T\rangle \ne \emptyset$.
Let $M$ be the hyperplane spanned by $L\cup T$. Since $\langle T\rangle$ does not contain the tangent line, $L$, of $C$ at $O$, we have
$\deg (A\cap T) \le 1$. Hence $\sharp (B\cap (M\setminus \{O\})) \ge 3d-1$. Hence $e_1\ge 3d$, a contradiction. Now, assume $L\cap \langle T\rangle =\emptyset$.
 We have $\deg (W_0\cap (L\cup T)) \ge 3d+1$. Since $O\in L$, we have $T\cap A=\emptyset$
and hence $A\nsubseteq T\cup L$. We have $h^0(\mathcal {I}_{L\cup T}(2)) = 7$ and $T\cup L$ is the scheme-theoretic base
locus $|\mathcal {I}_{L\cup T}(2)|$. Fix any $Q\in |\mathcal {I}_{L\cup T}(2)|$. Since $\deg (\mbox{Res}_Q(W_0)) \le d$, either $h^1(\mathcal {I}_{\mbox{Res}_Q(W_0)}(d-2)) =0$
or $\deg (\mbox{Res}_Q(W_0))=d$ and there is a line $J\supset \mbox{Res}_Q(W_0)$. In the former
case we have
 $W_0\subset Q$ and hence $W_0\subset L\cup T$, a contradiction. Now assume $\deg (\mbox{Res}_Q(W_0))=d$ and $J\supset \mbox{Res}_Q(W_0)$ for some line
 $J$. Since $\deg (A\cap J)\le 2$, we have $J\cap (B\setminus B\cap Q)\ne \emptyset$. Hence $J\ne L$ and $J$ is not a component of $T$. Since $L\cap \langle T\rangle =\emptyset$, we have $\sharp (B\cap (L\cup T \cup J)) \ge d-1+2d+d-2$, a contradiction.

\quad (b2.2.3.2.2) Assume the existence of $D$. Remark \ref{uno.1} gives $O\in D\cap L$. Since $f_1\le e_1\le 3d$ and $g'=1$, either there is a line $R'\subset N_1$ with $\deg (R'\cap Z_0)
\ge d+2$ or there is a conic $T'\subset N_1$ with $\deg (Z_0\cap T')\ge 2d+2$ or there is a plane cubic $F'\subset N_1$ with $\deg (Z_0\cap F') \ge 3d$
(Lemma \ref{w1}). The last case cannot occur, because it would give $f_1\ge 3d+1$ and hence $e_1\ge 3d+1$.

\quad (b2.2.3.2.2.1) Assume the existence of $R'$. Remark \ref{uno.1} gives that $R'$ is the tangent line of $C$ at $O$. Since $\deg (L\cap W_0) =d+1$, we have $R'\ne L$. Since $R'\subset N_1$ and $D\cap (B\setminus B\cap N_1)\ne \emptyset$,
we get $D\ne R'$. Remark \ref{uno.1} gives $O\in R'$ and that $R'$ is the tangent line of $C$ at $O$. Let $M\subset \mathbb {P}^4$ be any hyperplane containing
$D\cup L	\cup R'$. Since at most one of the lines $L, D, R'$ is tangent to $C$ at $O$, we have $\sharp ((B\setminus \{O\})\cap M) \ge 3d$ and $e_1\ge
\deg (W_0\cap M)\ge 3d+2$, a contradiction.

\quad (b2.2.3.2.2.2) Assume the existence of $T'$. If $O\notin T'$, then we get $\sharp (B\cap (T'\cup L\cup R)) \ge 4d-2$, a contradiction. Hence $O\in T'$.
Since $\sharp (B) \le 4d-4$, we also get that $\langle T'\rangle$ contains the tangent line to $C$ at $O$. We have $\sharp (B\cap (T\setminus \{O\})) \ge 2d-3$. If either $L\subset \langle T'\rangle$
or $R\subset \langle T\rangle$, then we get $e_1>\deg (W_0\cap \langle T\rangle ) \ge 3d$, a contradiction. If $L\nsubseteq \langle T'\rangle$
and $R\nsubseteq \langle T\rangle$, then neither $L$ nor $R$ is tangent to $C$ at $O$ and hence $\sharp ((B\setminus \{O\})\cap R) \ge d$ and $\sharp ((B\setminus \{O\})\cap L) \ge d$. Therefore $\sharp (B\setminus \{O\})\ge 4d-3$, a contradiction.

\quad (b3) Assume $g=1$. Since $A\nsubseteq H_1$, \cite[Lemma 5.1]{bb2} gives $h^1(\mathcal {I}_{W_1}(d-1)) >0$ and in particular
$\deg (W_1) \ge d+1$. Hence $e_1\le 3d$. By Lemma \ref{w1} and Remark \ref{w2} either there is a line $L\subset H_1$ such that $\deg (L\cap W_0) \ge d+2$
or there is a reduced conic $T\subset H_1$ such that $\deg (T\cap W_0) \ge 2d+2$
or there is a plane cubic $F\subset H_1$ with $\deg (W_0\cap F) \ge 3d$.

\quad (b3.1) Assume the existence of $F$. Since $A$ spans $\mathbb {P}^4$, we get $e_1> \deg (F\cap W_0) \ge 3d$, a contradiction.

\quad (b3.2) Assume the existence of $T$. Since $T$ is a plane curve, we have $\deg (T\cap A)\le 3$ and hence $\sharp (T\cap (B\setminus \{O\}) \ge 2d-1$. Remark \ref{uno.1} for $t=2$ gives $O\in T$ and that $T$ and $C$ have the same tangent line
at $O$. Since $A\nsubseteq H_1$, \cite[Lemma 5.1]{bb2} gives
$h^1(\mathcal {I}_{W_1}(d-1)) >0$ and hence $e_1 \le 4d+1-\deg (W_1) \le 3d$. Since $A$ spans $\mathbb {P}^4$, we get $e_1> \deg (W_0\cap T)$. Hence $\deg (W_1) \le 2d-2$. Therefore there is a line
$R\subset \mathbb {P}^4$ such that $\deg (R\cap W_1) \ge d+1$. Since $R\cap (B\setminus B\cap H_1)\ne \emptyset$, we have $R\nsubseteq H_1$ and
in particular $R\nsubseteq \langle T\rangle$. Hence $R$ is not tangent to $C$ at $O$. Remark \ref{uno.1} gives $O\in R$ and $\sharp (R\cap (B\setminus \{O\})) =d$.
Hence $M:= \langle T\cup R\rangle$ is a hyperplane, $e_1\ge \deg (M\cap W_0) \ge \deg (T\cap W_0) +\sharp (R\cap (B\setminus \{O\}))\ge 3d+2$, a contradiction.

\quad (b3.3) Assume the existence of $L$. Since $A$ spans $\mathbb {P}^4$ we have $e_1\ge d+4$. Remark \ref{uno.1} gives $O\in L$, that $L$ is the tangent line of $C$ at $O$ and that $\sharp (B\cap (L\setminus \{O\})) =d$ (it also gives $O\notin B$). Since $A\nsubseteq H_1$, \cite[Lemma 5.1]{bb2} gives $h^1(\mathcal {I}_{W_1}(d-1)) >0$.
We have $e_1+\deg (W_1) \le 4d+1$ and in particular $\deg (W_1) \le 3d-3$. Lemma \ref{w1} gives either the existence of a line $R\subset \mathbb {P}^4$ with
$\deg (R\cap W_1)\ge d+1$ or the existence of a conic $T'\subset \mathbb {P}^4$ with $\deg (T'\cap W_1) \ge 2d$ or the existence of a plane cubic $F$
with $\deg (F\cap W_1)\ge 3d-3$. The latter case is impossible, because it would give $e_1\ge 3d-2$ and hence $\deg (W_1)\le d+4$. Since $\deg (W_1)\ge d+1$, then
$e_1\le 3d$.

\quad (b3.3.1) Assume the existence of the line $R$. Since $R\cap (B\setminus B\cap H_1) \ne \emptyset$, we have $R\nsubseteq H_1$ and hence $R\ne L$.
Since $L$ is the tangent line of $C$ at $O$, we have $\deg (R\cap A)\le 1$. Remark \ref{uno.1} gives $O\in R$, $O\notin B$ and $\sharp (B\cap R) =d$. Let
$M\subset \mathbb {P}^4$ be a hyperplane containing $R\cup L$ and with maximal $h:= \deg (M\cap W_0)$. Since $A\nsubseteq M$, we have
$h^1(\mathcal {I}_{\mbox{Res}_M(W_0)}(d-1)) >0$. We have $h \ge 2d+3$ and hence
$\deg (\mbox{Res}_M(W_0)) \le 2d-2$. Therefore there is a line $D\subset \mathbb {P}^4$ such that $\deg (D\cap \mbox{Res}_M(W_0))\ge d+1$.
Since $D\cap (B\setminus B\cap M) \ne \emptyset$, we have $D\nsubseteq \langle R\cup L\rangle$. Since $L$ is the tangent line
of $C$ at $O$, Remark \ref{uno.1} gives $O\in D$ and $\sharp (D\cap B) =d$. Let $N$ be the hyperplane spanned by $D\cup R\cup L$.
Since $O\notin B$, we have $e_1\ge \deg (M\cap W_0) \ge d+2+d+d$, contradicting the inequality $e_1\le 3d$ proved in step (b3.3).

\quad (b3.3.2) Assume the existence of the conic $T'$ but the non-existence of any line $R$ with $\deg (R\cap W_1)\ge d$. If $T$ is reducible,
then $\deg (J\cap W_1)=d$ for each component $J$ of $T'$ (Remark \ref{w2}). Since $A$ spans $\mathbb {P}^4$, we have $e_1> \deg (T'\cap W_1)$. Therefore $e_1=2d+1$,
$\deg (T'\cap W_1) =2d$ and $W_1\subset T'$.  Since $e_1\le 3d-2$, we have
$L\cap \langle T'\rangle = \emptyset$. We have $h^0(\mathcal {I}_{L\cup T'}(2)) \ge 7$. Let $Q$ be any quadric hypersurface such that
$\deg (W_0\cap Q)> \deg (W_0\cap (L\cup T'))$. Since $\deg (\mbox{Res}_Q(W_0)) \le 4d+1-(3d+2)$, we have $h^1(\mathcal {I}_{\mbox{Res}_Q(W_0)}(d-2)) =0$
and hence $A\subset Q$ (\cite[Lemma 5.1]{bb2}). Since $L\cap \langle T'\rangle = \emptyset$, then $L\cup T'$ is the scheme-theoretic base-locus of $|\mathcal {I}_{L\cup T'}(2)|$. Since $A$ is connected, we get $A\subset L$, a contradiction.

\quad (c) Proposition \ref{w3} covers the case $b_1=2$ and $b_2=b_3=b_4=1$. Proposition \ref{w5.0} covers the case $b_1=b_2 =2$ and $b_3=1$. Proposition \ref{w5} covers the case $b_1=3$ and $b_2=b_3=1$. Proposition \ref{w4} covers the case $b_1=3$ and $b_2 =2$.

\quad (d) Now assume $s=2$, $b_1=4$, $b_2=1$. Assume $\sharp (B) \le 3d-2$ and hence $\deg (W_0) \le 3d+3$. Write $A = A_1\sqcup \{O_2\}$ and set $\{O_1\}:
=(A_1)_{red}$. We repeat the proof of Proposition \ref{w4}, except that now $\{O_2\}$ is the unique reduced connected component of $A$ and hence we cannot freely use \cite[Lemma 5.1]{bb2}; however,
$\deg (W_0) <3d+4$. We write (++xxxx) for the set-up of step (xxx) of the proof of Proposition \ref{w4} and use the notation of that step.

\quad (++b1.1.1) Assume $h^1(\mathcal {I}_{\mbox{Res}_M(W_0)}(d-1)) =0$. We get $O_2\in B$ and $A_1\cup B _1\subset M$, where $B_1:= B\setminus \{O_2\}$.
Since $P\ne \nu _d(O_2)$, we get that $\langle \nu _d(A_1)\rangle \cap \langle \nu _d(B_1)\rangle \ne \emptyset$, that $\nu _d(O_2)\notin \langle \nu _d(A_1)\rangle \cap \langle \nu _d(B_1)\rangle$ and that $\langle \nu _d(A)\rangle \cap
\langle \nu _d(B)\rangle$ is the linear span of $\nu _d(O_2)$ and $\langle \nu _d(A_1)\rangle \cap \langle \nu _d(B_1)\rangle$.
Hence there is a unique $P_1\in \langle \nu _d(A_1)\rangle \cap \langle \nu _d(B_1)\rangle$ such that $P\in \langle \{P_1,O_2\}$. Since $P\notin \langle W\rangle$
for any $W\subseteq A$ and any $P\subsetneq W$, we get $P_1\notin \langle A'\rangle$ for any $A'\subsetneq A'$ and any $A'\subsetneq B$.
Since $A_1$ evinces the scheme-rank of $P_1$, \cite[Proposition 5.19]{bb2} gives $r_{m,d}( P_1) = 3d-2$.

\quad (++b1.2) If $A\cup B\subset Q$, then that proof works. Assume $A\cup B\nsubseteq Q$. We get $O_2\in B$ and $A_1\cup B_1\subset Q$, where
$B_1:= B\setminus \{O_2\}$. Varying $Q$ we get $A_1\cup B_1\subset F\cup L$. Since $F\cap L =\emptyset$ and $\langle F\rangle$ is a plane, we have
$A_1\nsubseteq F\cup L$, a contradiction.

\quad (++b2) Assume $h^1(\mathcal {I}_{Z_1}(d-1)) =0$. We repeat step (++b1.1.1) with the hyperplane $N_1$ instead of the hyperplane $M$.

\quad (++c) Assume $h^1(\mathcal {I}_{W_1}(d-1)) =0$. We repeat step (++b1.1.1) with the hyperplane $H_1$ instead of the hyperplane $M$.

\quad (++c1) Assume $h^1(\mathcal {I}_{\mbox{Res}_M(W_0)}(d-1)) >0$. We repeat step (++b1.1.1).\end{proof}

\begin{proposition}\label{w3}
Fix integers $d, s, m $ such that $0 \le s \le m-1$, and $d\ge m$. Fix a degree $2$ connected zero-dimensional scheme $A_1\subset \mathbb {P}^m$ and a set
$S\subset \mathbb {P}^m$ such that $S\cap A_1=\emptyset$, $\sharp (S) = s$ and the scheme $A:= A_1\cup S$ is linearly independent. Set $\{O\}:= (A_1)_{red}$.
Fix any $P\in \langle \nu _d(A)\rangle$ such that $P\notin \langle \nu _d(A')\rangle$ for any $A'\subsetneq A$.
Then $r_{m,d}( P) =s+d$ and every $B\in \mathcal {S} ( P)$ is the disjoint union of $S$ and $d$ points of $\langle A_1\rangle \setminus \{O\}$.
\end{proposition}

\begin{proof}
By concision (\cite[Exercise 3.2.2.2]{l}) it is sufficient to do the case $\langle A\rangle =\mathbb {P}^m$, i.e. the case $m =s+1$. Since the case $m=1$ is true by Sylvester's theorem
(\cite{cs}),
we may assume $m\ge 2$ and use induction on $m$.
Since $h^1(\mathcal {I}_A(d)) =0$ and $A$ has only finitely many proper subschemes, $P$ exists and there is a unique point $Q\in \langle \nu _d(A_1)\rangle \setminus \{O\}$ such that $P\in \langle \{Q\}\cup \nu _d(S)\rangle$. Since $r_{1,d}(Q) =d$ by a theorem of Sylvester, we get  $r_{m,d}( P) \le s+d$. Fix $B\in \mathcal {S}( P)$. 
Set $W_0:= A\cup B$. We have $\deg (W_0) \le d+2s+2 =d+2m$. Since $A$ is not reduced, we have $A\ne B$ and hence $h^1(\mathcal {I}_{W_0}(d)) >0$. Take $H_i,e_i,W_i, g$ as in the proof of Theorem \ref{i2}. Since
$A$ spans $\mathbb {P}^m$, if $e_i\le m-1$, then $e_{i+1}=0$ and $W_i=\emptyset$. Hence $g\le d$. If $2\le g<d$ we also get $e_i \ge d+m+2-g$ for $i<g$
and hence $d+2m \ge g(d+m+2-g) -m+1$ and hence $g\le 2$. If $g=d$, then $e_d\ge 3$, $e_i\ge m$ if $i<d$ and hence $\deg (W_0)\ge m(d-1)+3$, a contradiction.

\quad (a) Assume $g=2$. Set $k:= \dim (\langle W_1\rangle )$. Remark \ref{w3.0} for the integer $d-1$
gives $e_2\ge k+d$. Since $A$ spans $\mathbb {P}^4$ we have $e_1\ge e_2+m-k \ge m+d$. Hence $2d+m +k\le d+2m$, contradicting the inequality $d\ge m$.

\quad (b) Assume $g=1$. Since $A$ spans $\mathbb {P}^m$, Remark \ref{w3.0} gives $e_1\ge d+m+1$.

\quad (b1) Assume $h^1(\mathcal {I}_{W_1}(d-1))=0$. We get $A_1\subset H_1$ and $S\setminus S\cap H_1 = B\setminus B\cap H_1$ (\cite[Lemma 5.1]{bb2}). Since
$A$ is in linearly general position and spans $\mathbb {P}^m$, we get that $S\setminus S\cap H_1$ is a point; call it $Q$.
Since $h^1(\mathcal {I}_{W_1}(d-1))=0$, Grassmann's formula gives that $\langle \nu _d(A)\rangle \cap \langle \nu _d(B)\rangle$ is the linear span of $\nu _d(Q)$ and $\langle 
\nu _d(A\setminus \{Q\}) \rangle \cap \langle 
\nu _d(B\setminus \{Q\}) \rangle$. Fix $P_1\in \langle 
\nu _d(A\setminus \{Q\}) \rangle \cap \langle 
\nu _d(B\setminus \{Q\}) \rangle$ such that $P\in \langle \{\nu _d(Q),P_1\}\rangle$. Since $P\notin \langle A'\rangle$ for any $A'\subsetneq A$, we get $P_1\notin
\langle \nu _d(A\setminus \{Q\})\rangle$. 
In the same way we get $P_1\notin \langle \nu_d(B')\rangle$ for any $B'\subsetneq B\setminus \{Q\}$. We may use the inductive assumption in $m$ and get
that $B$ is as claimed in the statement of Proposition \ref{w3}.

\quad (b2) Assume $h^1(\mathcal {I}_{W_1}(d-1)) >0$ and set $k:= \dim (\langle W_1\rangle )$. Remark \ref{w3.0} gives $\deg (W_1) \ge d+k$ and $e_1\ge d+m$.
and hence $d+k \le m$, a contradiction. 
\end{proof}

\begin{proposition}\label{w5.0}
Assume $d\ge 7$. Fix $O_3\in \mathbb {P}^m$, $m\ge 4$ and let $A_1, A_2$ be connected degree 2 subschemes of $\mathbb {P}^m$ such that
$A:= A_1\cup A_2\cup \{O_3\}$ spans a 4-dimensional linear subspace. Set $\{O_i\}:= (A_i)_{red}$. Fix $P\in \langle \nu _d(A)\rangle$ such that $P\notin \langle \nu _d(A')\rangle$
for any $A'\subsetneq A$. Then $r_{m,d}( P) =2d+1$. There are uniquely determined $P_i\in \langle \nu _d(A_i)\rangle$, $i=1,2$, such that $A_i$ evinces the cactus rank of $P_i$
and every $B\in \mathcal {S}( P)$ has a decomposition $B = B_1\sqcup B_2\sqcup \{O_3\}$ with $B_i\subset \langle A_i\rangle \setminus \{O_i\}$
$\sharp (B_i) =d$ and $B_i\in \mathcal {S}( P_i)$.
\end{proposition}

\begin{proof}
By concision (\cite[Exercise 3.2.2.2]{bb2}) we may assume $m=4$. Since $\nu _d(A)\rangle$ is linearly independent, there are unique $P_i \in \langle \nu _d(A_i)\rangle$, $i=1,2$,
such that $P\in \langle \{P_1,P_2,\nu _d(O_3)\}\rangle$. Since $P\notin \langle \nu _d(A')\rangle$ for any $A'\subsetneq A$, we have $P_i\ne \nu _d(A_i)$.
Hence $P_i$ has border rank 2 with respect to the degree $d$ rational normal curve $\nu _d(\langle A_i\rangle$. Since $P\in \langle \nu _d(B_1\cup B_2\cup \{O_3\})\rangle$
for all $B_i\in \mathcal {S}( P_i)$, we get $r_{m,d}( P) \le 2d+1$. Fix any $B\in \mathcal {S}( P)$ and set $W_0:= A\cup B$. We have $\deg (W_0) \le 2d+6$.
Define $H_i,e_i,W_i,g$ as in the proof of Theorem \ref{i2}. Since $\deg (W_0) \le 2d+6$, we get $g\le 2$.

\quad (a) Assume $g=2$. Since $e_1\ge e_2$ and $e_1+e_2\le 2d+6$, there is a line $L\subset H_2$ such that $\deg (L\cap W_1)\ge d+1$. Hence $e_1\le d+5$. Let $N_1\subset \mathbb {P}^4$
be a hyperplane containing $L$ and with maximal $f_1:= \deg (N_1\cap W_0)$. Define $N_i,f_i,Z_i,g'$ as in step (b1) of the proof of Theorem \ref{i2}. We get $g'\le 2$.

\quad (a1) Assume $g'=2$. Since $A$ spans $\mathbb {P}^4$ we have $f_1\ge \deg (L\cap W_1)+2\ge d+3$ and hence $f_2\le d+3$. Therefore
there is a line $R\subset N_2$ such that $\deg (Z_1\cap R)\ge d+1$. Since $\deg (R\cap A)\le 2$ and $Z_1\cap B \subseteq B\setminus B\cap L$, we get
$L\ne R$. Any hyperplane $M$ containing $L\cup R$ shows that $e_1\ge (d+1)+(d+1)-1$, a contradiction.

\quad (a2) Assume $g'=1$.

\quad (a2.1) Assume $\deg (L\cap W_0) \ge d+2$. The case $t=1$ of Remark \ref{uno.1} gives $\deg (L\cap A)=2$, $\sharp (B\cap L) =d$
and $L\cap A\cap B =\emptyset$. Since $A$ spans $\mathbb {P}^4$, we get $f_1\ge d+4$.

\quad (a2.1.1) Assume $h^1(\mathcal {I}_{Z_1}(d-1)) >0$. Since $\deg (Z_1) \le 2d+6-f_1\le 2(d-1)+1$, there is a line $R\subset \mathbb {P}^4$
such that $\deg (Z_1\cap R)\ge d+1$ and hence with $R\setminus R\cap L\ne \emptyset$ Hence $R\ne L$. Any hyperplane
$M$ containing $L\cup R$ has $\deg (M\cap W_0)\ge d+2+d+1-1$, contradicting the inequality $e_1\le d+5$.

\quad (a2.1.2) Assume $h^1(\mathcal {I}_{Z_1}(d-1)) =0$. Since $A\nsubseteq N_1$, \cite[Lemma 5.1]{bb2} gives $O_3\in B$ and $A_1\cup A_2\cup (B\setminus \{O_3\}
\subset N_1$. Since $e_1\le d+5$, we get $\sharp (B) \le d+2$. Since $\deg (W_0)\le d+7$, we immediately get $g=1$, a contradiction.

\quad (a2.2) Assume $\deg (L\cap W_0)=d+1$. Since $e_1\le d+5\le 2d+1$, there is a line $R\subset N_1$ such that $\deg (N_1\cap R)\ge d+2$
(\cite[Lemma 34]{bgi}). Since $R\ne L$,
we get $d+5\ge e_1\ge \deg (N\cap (L\cup R) \ge d+1+d+2-1$, a contradiction.

\quad (b) Assume $g=1$.

\quad (b1) Assume $h^1(\mathcal {I}_{W_1}(d-1)) >0$. We get $\deg (W_1) \ge d+1$ (\cite[Lemma 34]{bgi}) and hence $e_1\le d+5$. Since $h^1(H_1,\mathcal {I}_{H_1\cap W_0}(d))
>0$ and $e_1\le 2d+1$, there is a line $L\subset H_1$ such that $\deg (L\cap W_0)\ge d+2$. In particular $e_1\ge d+2$ and hence  $\deg (W_1) \le d+4$.
Since $2(d-1)+1 \le d+4$, we get the existence of a line $R\subset \mathbb {P}^4$ such that $\deg (W_1\cap R) \ge d+1$. Since $\deg (A\cap R)\le 2$, we
obtain $R\cap (B\setminus B\cap L)\ne \emptyset$ and hence $R\ne L$. Any hyperplane containing $L\cup R$ gives $e_1\ge (d+2)+(d+1)-1>d+5$, a contradiction.

\quad (b2) Assume $h^1(\mathcal {I}_{W_1}(d-1)) =0$. Since $A\nsubseteq H_1$, we get $O_3\in B$ and $A_1\cup A_2\cup (B\setminus \{O_3\})
\subset N_1$ (\cite[Lemma 5.1]{bb2}). Since $\nu _d(A)$ is linearly independent, there is a unique $Q\in \langle \nu _d(A_1\cup A_2)\rangle$ such that $P\in \langle \{Q\cup \nu_d(O_3)\}\rangle$ (we have $\{Q\} = \langle \nu _d(A_1\cup A_2)\rangle \cap \langle \{P,\nu _d(O_3)\}\rangle$). 
We get that $A_1\cup A_2$ evinces the cactus rank and the border rank of $Q$ and that $Q\in \langle \nu _d(B\setminus \{O_3)\rangle$.
Since $B\in \mathcal {S}( P)$, we get $B\setminus \{O\} \in \mathcal {S}(Q)$. Now we repeat the construction in $H_1$ with $A':= A_1\cup A_2$ and $B':= B\setminus \{O_3\}$.
From now on we work inside the 3-dimensional projective space $H_1$ and we use planes contained in $H_1$ as hyperplane. We use the notations $W_i, H_i,e_i,g,Z_i,f_i,g_i$
in this new set-up. For instance $W_0 =(A_1\cup A_2)\cup (B\setminus \{O\})$ and with respect to the point $Q$ instead of the point $P$. We need to check that
$B\setminus \{O_3\}\in \mathcal {S}( Q)$ and describe $B\setminus \{O_3\}$. For the first part it is sufficient to apply \cite[Proposition 5.17]{bb2} (which gives $r_{m,d}( Q) \le 2d$), because we have $\sharp
(B\setminus \{O_3\}) \le 2d$. In particular we checked that $r_{m,d}( P) =2d+1$, i.e. the part of Proposition \ref{w5.0} used in Theorem \ref{i2}.
The only difference with respect to the proof of Theorem \ref{i2} is that (since we are working with planes) we only now that if $e_i\le 2$, then $e_{i+1} =0$ and $W_i=\emptyset$
and that if $f_i\le 2$, then $f_{i+1} =0$ and $Z_i=0$. We have $w_0:= \deg (A_1\cup A_2\cup (B\setminus \{O_3\}))\le 2d+4$. First assume $g \ge d+2$ (resp. $g'\ge d+1$).
We get $w_0 \ge 1+3(d+1)$, a contradiction. Now assume $g=d+1$ (resp. $g'=d+1$). We find $w_0>3d$, a contradiction. Now
assume $g<d$ and $e_g \ge 2(d+1-g)+2$ (resp. $g'<d$ and $f_{g'} \ge 2(d+1-g)+2$). We get $w_0\ge 2g(d+1-g) +4$ (resp. $w_0\ge 2g'(d+1-g')+4$) and hence $g=1$ (resp. $g'=1$). In particular if $H_1$ is a plane of $\mathbb {P}^3$ with $\deg (H_1\cap W_0)$ maximal, then $h^1(H_1,\mathcal {I}_{H_1\cap W_0}(d)) >0$ and hence (since
$H_1\cap W_0$ spans $H_1$) $e_1\ge d+3$ (Remark \ref{w3.0}).
Since $A_1\cup A_2$ spans $\mathbb {P}^3$ and no connected component of $A_1\cup A_2$ is reduced, \cite[Lemma 5.1]{bb2} gives $h^1(\mathcal {I}_{W_1}(d-1)) >0$
and in particular $\deg (W_1) \ge d+1$. Hence $e_1\le d+3$, which implies $\deg (W_1) = d+1$ and the existence
of a line $R$ such that $W_1\subset R$. Since $e_1\le d+3$, there is a line $L\subset H_1$ such that $\deg (L\cap W_0) \ge d+2$. Remark \ref{uno.1} gives that
either $L =\langle A_1\rangle$ or that $L =\langle A_2\rangle$, $\sharp (B\cap L) =d$ and $B\cap (A_1\cup A_2)\cap L) =\emptyset$. Up to renaming $A_1$ and $A_2$
we may assume $L =\langle A_1\rangle$. Remark \ref{uno.1} gives $O_1\notin B$ and $\sharp (B\cap L) =d$,
i.e. $B\cap \langle A_1\rangle$ is as described in the statement of Proposition \ref{w5.0}. Remark \ref{uno.1} gives $O_2\in R$ and that either $O_2\notin H_1$, $R =\langle A_2\rangle$ and $\sharp (B\cap W_1)=d-1$
or $O_2\in H_1\cap R$, $R =\langle A_2\rangle$ and $\sharp (B\cap (R\setminus \{O_2,O_3\}))= d+1$. In the latter case we get $O_2\notin B$, because
we know that $\sharp (B)\le 2d+1$. Hence in the latter case $B \setminus \{O_3\} \subset \langle A_1\rangle \cup \langle A_2\rangle$
and each $B\cap \langle A_i\rangle$ is as described in Proposition \ref{w5.0}. Now assume $O_2\notin H_1$. Recall that $R =\langle A_2\rangle$ and $L=\langle A_1\rangle$.
Since $e_1=d+3$, there is
a unique $o\in (B\setminus \{O_3\})\cap H_1$ such that $H_1 = \langle L\cup \{o\}\rangle$. Fix any $Q\in (B\setminus B\cap H_1)$ and set $M:= \langle \{Q\}\cup L\rangle$.
Since $e_1=d+3$, we have $o\notin M$. Since $A_1\cup A_2$ has no reduced connected component, \cite[Lemma 5.1]{bb2} gives $h^1(\mathcal {I}_{\mbox{Res}_M(W_0)}(d-1))>0$.
We first get the existence of a line $R'\subset \mathbb {P}^3$ such that $\deg (R' \cap  \mbox{Res}_M(W_0))\ge d+1$
and then $\deg (\mbox{Res}_M(W_0))=d+1$, $\mbox{Res}_M(W_0)\subset R'$ and $R' =R$. We get $o\in \langle A_2\rangle$. Hence $B\cap \langle A_2\rangle$
is as described in the statement of Proposition \ref{w5.0} and $B\setminus \{O_3\} \subset \langle A_2\rangle \cup \langle A_1\rangle$.
\end{proof}

For a description of $\mathcal {S}(P_1)$ and $\mathcal {S}(P_2)$ below see \cite[part (b3) of Theorem 4 and part (f) of \S 4]{b} and\cite[Proposition 1 and Theorem 3]{b}, respectively.
The case in which the conic $C\supset A_1\cup B_1$ is reducible is described in \cite[Proposition 7]{b}.

\begin{proposition}\label{w4}
Fix an integer $d\ge 9$. Let $A_1\subset \mathbb {P}^4$ be a degree $3$ connected curvilinear scheme spanning a plane and $A_2\subset \mathbb {P}^4$ be a degree two connected
curvilinear scheme. Assume $\langle A_1\rangle \cap \langle A_2\rangle =\emptyset$, i.e. assume $A_1\cap A_2=\emptyset$ and $\dim (\langle A_1\cup A_2\rangle )=4$.
Then $r_{m,d}( P) = 3d-1$. Fix any $B\in \mathcal {S}( P)$. Then there
is a unique decomposition  $B = B_1\sqcup B_2$ with $B\cap A_{red }=\emptyset$, $\sharp (B_1) = 2d-1$, $\sharp (B_2)=d$,
$B_1\subset \langle A_1\rangle $, $B_2\subset \langle A_2\rangle$, each $B_i$ evincing the symmetric tensor
rank of a uniquely determined point of $P_i\in \langle \nu _d(A_i)\rangle$.The scheme $A_1\cup B_1$ is contained in a reduced conic $C\subset \langle A_1\rangle$.

 Conversely, for any $P_i\in \langle \nu _d(A_i)\rangle$, $i=1,2$,
such that $P_i\notin \langle \nu _d(A')\rangle$ for any $A'\subsetneq A_i$, any $B_i\in \mathcal {S}(P_i)$ and any $P\in \langle \{P_1,P_2\}\rangle \setminus \{P_1,P_2\}$
we have $B_1\cap B_2=\emptyset$ and $B_1\cup B_2\in \mathcal {S}( P)$. 
\end{proposition}

\begin{proof}
Since $P\in \langle \nu _d(A)\rangle$ and $\langle \nu _d(A_1)\rangle \cap \langle \nu _d(A_2)\rangle =\emptyset$, there are uniquely determined
$P_i\in \langle \nu _d(A_i)\rangle$, $i=1,2$, such that $P\in \langle \{P_1,P_2\}\rangle$. Since $P\notin \langle \nu _d(A')\rangle$ for any $A'\subsetneq A$, then $P\ne P_i$
and
$P_i\notin \langle \nu _d(A')\rangle$ for any $A'\subsetneq A$. Every element of $\mathcal {S}(P_i)$ is contained in $\langle A_i\rangle$ (\cite[Exercise 3.2.2.2]{l})
and each element of $\mathcal {S}( P_i)$ is known (\cite[Proposition 1 and Theorem 2]{b} for $i=1$, \cite[Theorem 3 and \S 4]{b} for $i=2$). Since $\langle A_1\rangle \cap \langle A_2\rangle =\emptyset$, we
have $B_1\cap B_2=\emptyset$ for all $B_i\in \mathcal {S}( P_i)$.

Set $\{O_1\} =(A_1)_{red}$ and $\{O_2\} = (A_2)_{red}$. Fix $B\in \mathcal {S}( P)$. Since $r_{m,d}(P_1) =2d-1$ and $r_{m,d}(P_2) =d$,
we have $r_{m,d}( P) \le 3d-1$. Hence to prove all the assertions of the proposition it is sufficient to prove that $B = B_1\sqcup B_2$ with $B_i\in \mathcal {S}( P_i)$ for all $i$. Set $W_0:= A\cup B$.
We have $w_0:= \deg (W_0) \le 3d+4$. Define $e_i, H_i, W_i,g$ as in step (b) of the Theorem \ref{i2}. Since $w_0\le 3d+4$, we get $g\le 3$.

\quad (a) Assume $g=3$. Since $e_1\ge e_2\ge e_3$ and $e_1+e_2+e_3 \le 3d+4$, we have $e_3\le d+1 \le 2(d-2)+1$. Since $h^1(H_3,\mathcal {I}_{W_2\cap H_3}(d-2))>0$,
there is a line $L\subset H_3$ such that $\deg (L\cap W_2)\ge d$. Since $A$ spans $\mathbb {P}^4$, the maximality property
of the integer $e_2$ gives $e_2 \ge 2 +\deg (L\cap W_1) \ge d+2$. Since $e_1\ge e_2$, we get $e_1=e_2=d+2$ and $e_3 = d$. Set $Z_0:= W_0$
and define $Z_i, N_i, f_i, g'$ as in step (b1) of the proof of Theorem \ref{i2}. As above we get $g'\le 3$ and $f_1=d+2$.

\quad (a1) Assume $g'=3$. As above we get $f_1=f_2=d+2$, $f_3=d$ and the existence of a line $R\subset N_3$ such that $Z_3\subset R$. Since $\deg (R\cap A)\le 2$, we have $R\cap (B\setminus B\cap N_1)
\ne \emptyset$ and hence $R\ne L$. Hence $\deg (W_0\cap (L\cup R)) \ge d+d-1>d+2 =e_1$, contradicting the existence of a hyperplane containing $L\cup R$. 

\quad (a2) Assume $g'=2$. Since  $f_2\le 3d+4-f_1$, Lemma \ref{w1} gives that either there is a line $D\subset N_2$ with $\deg (D\cap Z_1)\ge d+1$ or there
is a conic $F\subset N_2$ with $\deg (F\cap W_1)\ge 2d$. The latter case is impossible, because $e_1<2d$. Hence $D$ exists. Since $\deg (D\cap A)\le 2$, we have $D\cap (B\setminus B\cap N_1)
\ne \emptyset$ and hence $D\ne L$. Any hyperplane $M\supset D\cup L$ gives $e_1\ge (d+1)+d-1>d+2$, a contradiction.

\quad (a3) Assume $g'=1$. Since $f_1\le e_1=d+2$, \cite[Lemma 34]{bgi} gives the existence of a line $T\subset N_1$ such that $\deg (Z_0\cap T)\ge d+2$.
Since $A$ spans $\mathbb {P}^4$, we get $e_1 \ge 2 +\deg (Z_0\cap T)$, a contradiction.

\quad (b) Assume $g=2$. Since $e_1\ge e_2$ and $e_1+e_2\le 3d+4$, we have $e_2 \le 2(d-1)+1$. Hence \cite[Lemma 34]{bgi} gives the existence of a line $L\subset N_2$
such that $\deg (L\cap W_1)\ge d+1$. We define $Z_i,f_i,N_i,g'$ with respect to the line $L$ met in this step. We have
$f_1\ge \deg (L\cap W_1)+2 \ge d+3$. Hence $\sum _{i\ge 2} f_i \le 2d+1$, excluding the case $g'=3$. Therefore $1\le g'\le 2$.

\quad (b1) Assume $g'=2$. Since $f_2\le 2d+1$, either there is a line $R\subset N_2$ with $\deg (R\cap Z_1)\ge d+1$ or there is a reduced conic $F\subset N_2$ with $\deg (F\cap W_1)\ge 2d$ (Lemma \ref{w1} and Remark \ref{w2}).

\quad (b1.1) Assume the existence of the line $R\subset N_2$ with $\deg (R\cap Z_1)\ge d+1$. Since $L\subset N_1$ and $R\cap (B\setminus B\cap N_1)\ne \emptyset$, we have $R\ne L$. Let $M$ be a hyperplane containing $R\cup L$ and
such $h_1:= \deg (M\cap W_0)$ is maximal. If $R\cap L = \emptyset$, then $h_1\ge 2d+2$. If $R\cap L\ne \emptyset$, then $h_1\ge 1 +\deg (W_0\cap (L\cup R))
\ge 2d+2$.

\quad (b1.1.1) Assume $h^1(\mathcal {I}_{\mbox{Res}_M(W_0)}(d-1))=0$. Since no connected component of $A$ is reduced, \cite[Lemma 5.1]{bb2} gives
$A\cup B\subset M$, a contradiction.

\quad (b1.1.2) Assume $h^1(\mathcal {I}_{\mbox{Res}_M(W_0)}(d-1))>0$. Since $h_1\ge 2d+2$, we have $\deg (\mbox{Res}_M(W_0))\le d+2\le 2(d-1)+1$.
Hence there is a line $D\subset \mathbb {P}^4$ such that $\deg (D\cap \mbox{Res}_M(W_0)) \ge d+1$. Since $R\cup L\subset M$, the lines $D,L,R$ are different.
We have $h^0(\mathcal {I}_{D\cup L\cup R}(2)) \ge 6$. Fix a general $Q\in |\mathcal {I}_{D\cup L\cup R}(2)|$. We have
$\deg (Q\cap W_0)\ge \deg (W_0\cap R)+\deg (W_0\cap L)-1+\deg (W_0\cap D)-2$ and hence $h^1(\mathcal {I}_{\mbox{Res}_Q(W_0)}(d-2)) =0$.
Since no connected component of $A$ is reduced, \cite[Lemma 5.1]{bb2} applied to the degree two hypersurface $Q$ instead of a hyperplane
gives $W_0\subset Q$. Hence $W_0$ is contained in the scheme-theoretic base locus $\mathcal {B}$ of $|\mathcal {I}_{D\cup L\cup R}(2)|$. The scheme $\mathcal {B}$
depends from the mutual position of the 3 lines, $D, L, R$, but just knowing that $A\subseteq \mathcal {B}$ gives that $\mathcal {B}$ spans $\mathbb {P}^4$.
Hence either the 3 lines $L, D, R$ are disjoint and $\langle L\cup D\cup R\rangle =\mathbb {P}^4$ or two of these lines, say $J_1,J_2$ meets, while the
other one, $J_3$, is disjoint from the plane $\langle J_1\cup J_2\rangle$. The former case cannot occur, because $A_1$ is contained in no line.
Assume that $J_1\cup J_2$ is a conic and that $J_3\cap \langle J_1\cup J_2\rangle =\emptyset$. We have $\mathcal {B} = J_1\cup J_2\cup J_3$.
Since $A_1$ is not contained in a line, we get $A_1\subset J_1\cup J_2$ (and hence $A_2\subset J_3$) and $\{O_1\} =J_1\cap J_2$. Set $B_1:= B\cap (J_1\cup J_2)$
and $B_2:= B\cap J_3$.
Since $h^1(\mathcal {I}_{\langle J_1\cup J_2\rangle \cup J_3}(d)) =0$ and $J_3\cap \langle J_1\cup J_2 \langle =\emptyset$,
the projective space $\langle \nu_d(J_1\cup J_2\cup J_3)\rangle$ is spanned by its two subspaces
$\langle \nu_d(J_1\cup L_2)\rangle$ and $\langle \nu_d(J_3)\rangle$. Hence there are uniquely determined points
$U_1\in \langle \nu_d(J_1\cup J_2)\rangle$ and $U_2\in \langle \nu_d(J_3)\rangle$ such that $P\in \langle \{U_1,U_2\}\rangle$.
For the same reason there are uniquely determined points
$Q_1\in \langle \nu_d(A_1)\rangle$, $Q_2\in \langle \nu_d(A_2)\rangle$, $Q'_1\in \langle \nu_d(B_1)\rangle$, $Q'_2\in \langle \nu_d(B_2)\rangle$ such that $P\in \langle \{Q_1,Q_2\}\rangle$ and $P\in \langle \{Q'_1,Q'_2\}\rangle$. The uniqueness of $U_i$, $i=1,2$, gives $U_i=Q_i$ and $U_i =Q'_i$. Hence $Q_i=Q'_i$, $i=1,2$. Obviously
$A_i$ evinces the scheme-rank and the border-rank of $Q_i$. Since
$B$ evinces the rank of $P$, $B_i$ evinces the rank of $Q_i$. Since $(J_1\cup J_2)\cap J_3=\emptyset$, we also get $B_1\cap B_2 =\emptyset$.
Hence $B$ is as claimed in the statement of Proposition \ref{w4}.

\quad (b1.2) Assume the existence of the conic $F$ with $\deg (F\cap W_1)\ge 2d$, but that there is no line $R$ as in (b1.1). Since $L\subset N_1$, $W_1\cap B\cap N_1=\emptyset$ and $\deg (L\cap A)\le 2$,
we get $\deg (W_0\cap (L\cup F)) \ge 3d-2$. Since $e_2\ge d+1$, we have $e_1\le 3d+4-e_2 <3d-2$. Hence $\langle L\cup F\rangle =\mathbb {P}^4$, i.e. $L\cap
\langle F\rangle =\emptyset$. Therefore $L\cup F$ is the scheme-theoretic intersection of the linear system $|\mathcal {I}_{F\cup L}(2)|$. Fix a general $Q\in |\mathcal {I}_{F\cup L}(2)|$. Since $F\cap L=\emptyset$, we have $\deg (Q\cap W_0)\ge 3d+1$ and hence $\deg (\mbox{Res}_Q(W_0))\le 3 \le d-1$.
Therefore $h^1(\mathcal {I}_{\mbox{Res}_Q(F\cup L}(d-2)) =0$. Since no connected component of $A$ is reduced, \cite[Lemma 5.1]{bb2} gives
$A\cup B\subset Q$. Since $L\cup F$ is the scheme-theoretic intersection of the linear system $|\mathcal {I}_{F\cup L}(2)|$, we get
$A\cup B\subset L\cup F$. Since $L\cap \langle F\rangle =\emptyset$ and $A_1\nsubseteq F$, we get $A_2\subset L$ and $A_1\subset F$.
We repeat the last part of the proof of step (b1.1.1) with $F$ instead of $J_1\cup J_2$ and $L$ instead of $J_3$.

\quad (b2) Assume $g'=1$. If $h^1(\mathcal {I}_{Z_1}(d-1)) =0$, then \cite[Lemma 5.1]{bb2} gives $A\subset N_1$ (because
no connected component of $A$ is reduced), a contradiction. Hence we may assume $h^1(\mathcal {I}_{Z_1}(d-1)) >0$. Since  $f_1\ge d+3$, we have $\deg (Z_1) \le 2d+1$ and hence either there is a line $R\subset \mathbb {P}^4$ such that $\deg (R\cap Z_1) \ge d+1$
or there is a reduced conic $F\subset \mathbb {P}^4$ such that $\deg (F\cap Z_1) \ge 2d$. In the second case we repeat the proof of step (b1.2.1). In the first case we
repeat all steps of (b1.2).

\quad (c) Assume $g=1$. Since $A\nsubseteq H_1$ and no connected component of $A$ is reduced, \cite[Lemma 5.1]{bb2} gives $h^1(\mathcal {I}_{W_1}(d-1)) >0$.
Remark \ref{w3.0} gives $e_1\ge d+4$ and hence $\deg (W_1) \le 2d$. Hence either there is a line $R\subset \mathbb {P}^4$ with $\deg (R\cap W_1)\ge d+1$
or $\deg (W_1) =2d$, $e_1=d+4$ and there is a reduced conic $F\subset \mathbb {P}^4$ such that $W_1\subset F$. The second case is impossible, because
it would give $e_1\ge 2d+1$. Therefore there is a line $R\subset \mathbb {P}^4$ with $\deg (R\cap W_1)\ge d+1$. We also have $e_1\le 2d+3$. Since $h^1(H_1,\mathcal {I}_{N_1\cap W_0}(d)) >0$ and $e_1\le 2d+3$, either there is a line $L\subset H_1$ with $\deg (L\cap W_0)\ge d+2$ or there is a conic $F\subset H_1$
such that $\deg (F\cap W_0)\ge 2d+2$ (Lemma \ref{w1}).

\quad (c1) Assume the existence of a line $L\subset H_1$ with $\deg (L\cap W_0)\ge d+2$. Since $\deg (R\cap A)\le 2$, we have $R\cap (B\setminus B\cap H_1))\ne
\emptyset$ and hence $R\ne L$. Let $M\subset \mathbb {P}^4$ be a hyperplane containing $\langle L\cup R\rangle$ and with maximal $\deg (M\cap W_0)$.
If $L\cap R\ne \emptyset$, then $\deg (W_0\cap (L\cup R)) \ge 2d+2$ and hence $\deg (M\cap W_0)\ge 2d+3$. If $L\cap R =\emptyset$, then $\deg (M\cap W_0)
\ge 2d+3$. Since $h^1(\mathcal {I}_{\mbox{Res}_M(W_0)}(d-1)) >0$ by \cite[Lemma 5.1]{bb2}, we first get $\deg (\mbox{Res}_M(W_0))\ge d+1$
and then $e_1= \deg (W_0\cap M)=2d+3$, $\deg (\mbox{Res}_M(W_0))= d+1$ and the existence of a line $D\subset \mathbb {P}^4$ such
that $\mbox{Res}_M(W_0)\subset D$. See step (b1.1.1); in this case $J_3=L$, $\{J_1,J_2\} = \{R,D\}$. A posteriori we get $R\cap D =\{O_1\}$
and $A_2\subset L$; $B\cap (R\cup D)$ is as described in \cite[Proposition 7]{b}.

\quad (c2) Assume the existence of the conic $F$, but the non-existence of any line $L\subset H_1$ with $\deg (L\cap W_0)\ge d+2$. Remark \ref{w3.0}
gives that $F$ is reduced and that if $F =J_1\cup J_2$ with each $J_i$ is a line, then $\deg (J_i\cap W_0)=d+1$, $i=0,1$. Since $e_1\le 2d+3$
and $A$ spans $\mathbb {P}^4$, we get $e_1=2d+3$, $\deg (F\cap W_0)=2d+2$ and $\deg (W_0\cap \langle F\rangle )=2d+2$. Since $h^1(\mathcal {I}_{W_1}(d-1)) >0$
by \cite[Lemma 5.1]{bb2} and $\deg (W_1) \le d+1$, we get $\deg (W_1) =d+1$ and the existence of a line $J\subset \mathbb {P}^4$ such that
$W_1\subset J$. Since $\deg (J\cap A)\le 2$, we have $J\cap (B\setminus B\cap H_1)\ne \emptyset$ and hence $J\nsubseteq F$.
If $J\cap \langle F\rangle \ne \emptyset$, then $e_1\ge 2d+2+d+1-2$, a contradiction. Therefore $J\cap \langle F\rangle =\emptyset$. We go as in step (b1.2), i.e. as in the last part
of step (b.1.1.1).
\end{proof}

\begin{proposition}\label{w5}
Fix integers $s\ge 0$, $m\ge s+2$ and $d \ge 2m+1$. Fix a degree $3$ curvilinear zero-dimensional scheme $A_1\subset \mathbb {P}^m$ with $\dim (\langle A_1\rangle )=2$ and a set
$S\subset \mathbb {P}^m$ such that $S\cap A_1=\emptyset$, $\sharp (S) = s$, $S$ is linearly independent and $\langle S\rangle \cap \langle A_1 \rangle
=\emptyset$. Set $A:= A_1\cup S$.
Fix any $P\in \langle \nu _d(A)\rangle$ such that $P\notin \langle \nu _d(A')\rangle$ for any $A'\subsetneq A$. There is a unique $P_1\in \langle \nu _d(A_1)\rangle$
such that $P\in \langle \{P_1\}\cup \nu _d(S)\rangle$. We have $P_1\notin \langle \nu _d(A')\rangle$ for any $A'\subsetneq A$.
Then $r_{m,d}( P) =s+2d-1$ and every $B\in \mathcal {S} ( P)$ is of the form $B = B_1\sqcup S$ with $B_1\in \mathcal {S}( P_1)$ (and the converse holds).
The elements of $\mathcal {S}( P_1)$ are described in the case $m=2$ of \cite[part (b3) of Theorem 4 and part (f) of \S 4]{b}.
\end{proposition}

\begin{proof}
The last sentence follows from concision (\cite[Exercise 3.2.2.2]{l}). 

The scheme $A$ is linearly independent and $\deg (A) =3+s$. Hence $\nu _d(A)$ is linearly independent. Every element of $\mathcal {S}( P)$ is contained in $\langle A\rangle$ (\cite[Exercise 3.2.2.2.]{l}). Hence
it is sufficient to do the case $s=m-2$. Fix any $B\in \mathcal {S} ( P)$. Since $P\in \langle \nu _d(E\cup S)\rangle$ for any $E\in \mathcal {S}( P_1)$, we have
$r_{m,d}( P) \le m+2d-3$ and it
is sufficient to prove that $B = S\sqcup B_1$ with $B_1\in \mathcal {S}( P_1)$. Set $\{O\} := (A_1)_{red}$ and let ${\bf{w}}$ the degree two
subscheme of $A_1$.
If $m=2$, then this proposition is the case $m=2$ of \cite[part (b3) of Theorem 4]{b}. Hence we may assume
$m>2$ and use induction on $m$. Set $W_0:= A\cup B$ and define $H_i$, $e_i$ and $g$ as in step (b) of the proof 
of Theorem \ref{i2}. If $e_i\le m-1$, then $e_{i+1}=0$ and $W_i=\emptyset$. We have $w_0:= \deg (W_0) \le 2d+2m-2$. Assume for the moment $g =d+1$. We get $h^1(\mathcal {I}_{W_d\cap H_{d+1}}) >0$, i.e.
$e_{d+1} \ge 2$. Since $e_{d+1}>0$, we have $e_i\ge m$ for $i\le d$ and hence $2d+2m-2 \le md+2$, which is absurd if $m>2$; for $m=2$ we have $s=0$, $P_1=P$
and the statement is a tautology. Assume $g\ge d+2$. We get $W_{d+1} \ne \emptyset$ and hence $2d+2m-2\ge 1 +m(d+1)$, a contradiction. Now
assume $2\le g \le d$. Using Remark \ref{w3.0} we get $2d+2m-2 \ge g(d+1+m-g) -m+2$. Since $d\ge 2m+1$, we get $g\le 2$. Hence $1\le g\le 2$.

\quad (a) Assume $g=2$. Since $e_1\ge e_2$ and $4d> 2d+m-2$, we have $e_1\le 2d-1$.
By \cite[Lemma 34]{bgi} there is a line $L\subset H_2$ such that $\deg (L\cap W_1) \ge d+1$. Since $e_2\ge d+1$, we have $e_1\le d+2m-3$. Set $Z_0:= W_0$.
Set $k:= \dim (\langle H_2\cap W_1\rangle )$. If $k\le m-2$, then the maximality property of the integer
$e_2$ gives $W_2=\emptyset$. Remark \ref{w1} for the integer $d-1$ gives $e_2 \ge d+k$. Since $e_1\ge e_2$, if $k=m-1$ we have $e_1+e_2=2d+2m-2$.
Hence in all cases we have $W_2=\emptyset$. Let $M\subset \mathbb {P}^m$ be a hyperplane containing $\langle H_2\cap W_1\rangle$ and with maximal $\deg (W_0\cap M)$
(we have $M= H_2$ if $k=m-1$). Since $A$ spans $\mathbb {P}^m$, we have $\deg (W_0\cap M) \ge d+m-1$ and hence $\deg (\mbox{Res}_M(W_0))\le d+m-1$.

\quad (a1) Assume $h^1(M,\mathcal {I}_{M\cap W_0}(d)) =0$. Set $Z_0:= W_0$. A residual exact sequence gives $h^1(\mathcal {I}_{\mbox{Res}_M(W_0)}(d-1)) >0$. Since $\deg (\mbox{Res}_M(W_0))\le d+m-1 \le 2d-1$, there is a line $L\subset \mathbb {P}^m$ such that $\deg (L\cap \mbox{Res}_M(W_0))\ge d+1$. Let $N_1\subset \mathbb {P}^m$ be a hyperplane
such that $N_1\subset L$ and $f_1:= \deg (N_1\cap W_0)$ is maximal and set $Z_1:= \mbox{Res}_{N_1}(Z_0)$. Since $A$ spans $\mathbb {P}^m$, then $f_1\le d+m-1$ and hence $\deg (Z_1) \le d+m-2$. For all integers $i\ge 2$ define $f_i,Z_i,N_i$ as in step (b1) of the proof of Theorem \ref{i2}.
We have $f_i\ge f_{i+1}$ for all $i\ge 2$ and $f_{i+1}=0$ if $f_i\le m-1$.
By residual exact sequence like (\ref{eqc1}) we get the existence of a minimal integer $g'>0$ such that $h^1(\mathcal {I}_{W_{g'-1}\cap N_{g'}}(d+1-g')) >0$.
As in step (b1) of the proof of Theorem \ref{i2} we get $g'\le 2$.

\quad (a1.1) Assume $g'=2$. Since $d+m-2 \le 2(d-2)+1$, there is a line $R\subset \mathbb {P}^m$ such that $\deg (R\cap Z_1)\ge d+1$. Since $\deg (R\cap A)\le 2$,
we have $\deg (R\cap Z_1\cap B) >0$. Since $Z_1\cap B =B\setminus B\cap N_1$ and $L\subset N_1$, we have $R\ne L$. First assume  that either
$L\cap R\ne \emptyset$ or $m\ge 4$ (i.e. assume that $L\cup R$ is contained in a hyperplane of $\mathbb {P}^m$). Since $\deg (L\cap R)=1$ as schemes,
we get $e_1\ge d+1+d+1-1$. Hence $2d+1 \le d+2m-3$, contradicting our assumption $d \ge 2m+1$. Now assume $m=3$ and $L\cap R =\emptyset$.
Since $A_1$ is connected and spans a plane, we have $A_1\nsubseteq L\cup R$.
Fix a general $Q\in |\mathcal {I}_{L\cup R}(2)|$. Since $A_1$ is curvilinear and $L\cup R$ is the scheme-theoretic base locus of the linear system $|\mathcal {I}_{L\cup R}(2)|$,
we have $A_1\nsubseteq Q$. Since $\deg (\mbox{Res}_Q(W_0)) \le 2d+2m-2 -(d+1)-(d+1) \le d-1$ (by our assumption $d\ge 2m+1$), we have $h^1(\mathcal {I}_{\mbox{Res}_Q(W_0)}(d-2)) =0$. By \cite[Lemma 5.1]{bb2} applied to the degree two surface $Q$ instead of a plane
we get $A_1\subset Q$, a contradiction.

\quad (a1.2) Assume $g'=1$.

\quad (a1.2.1) Assume $h^1(M,\mathcal {I}_{M\cap W_0}(d)) =0$.  By \cite[Lemma 5.1]{bb2} we have $A_1\subset M$ and $S\setminus S\cap M = B\setminus B\cap M$.
Since $A$ spans $\mathbb {P}^m$, we have $S\ne S\cap M$. Set $E:= S\cap B$. We just saw that $E\ne \emptyset$ and that $(W_0\setminus E)\subset M$.
Since $E\subset S$, we have $h^1(\mathcal {I}_E(d-1)) =0$. Hence a residual exact sequence and Grassmann's formula
gives that $\langle \nu _d(A)\rangle \cap \langle \nu _d(B)\rangle$ is spanned by its subspaces $\langle \nu _d(A\setminus E)\rangle \cap \langle \nu _d(B\setminus E)\rangle$
and $\langle \nu _d(E)\rangle$ and that these subspaces are disjoint. Take any $Q_1\in \langle \nu _d(A\setminus E)\rangle \cap \langle \nu _d(B\setminus E)\rangle$
such that $P\in \langle \{Q_1\}\cup \nu _d(E)\rangle$. Since $P\notin \langle \nu _d(A')\rangle$ for any $A'\subsetneq A$, we
have $Q_1\notin \langle \nu _d(A')\rangle$ for any $A'\subsetneq A$. $Q_1$ is the only point of $\langle \nu _d(A\setminus E)\rangle$
such that $P\in \langle \nu _d(E)\cup \{Q_1\}\rangle$ and $P_1$ is associated also to $Q_1$. By induction on $m$ (or \cite{b} if $m-\sharp (E) =2$) we get $r_{m,d}(Q_1) = 2d+m-3 -\sharp (E)$
and that every $U\in \mathcal {S}( Q_1)$ is of the form $U\sqcup S\setminus E$ with $\sharp (U)=2d-1$ and $U\in \mathcal {S}( P_1)$.
See Step (b2) below.

\quad (a1.2.1) Assume $h^1(M,\mathcal {I}_{M\cap W_0}(d)) >0$. Since $\deg (M\cap W_0)\le e_1\le 2d-1$, there is a line $R\subset M$ such that $\deg (R\cap W_0)\ge d+2$.

\quad (a1.2.1.1) Assume $R =L$. Take $N_i, Z_i, f_i$ as in step (a1). Now it is easier, because $f_1\ge d+m$. 

\quad (a1.2.1.2) Assume $R\ne L$. First assume that either $m\ge 4$ or $m=3$ and $R\cap L\ne \emptyset$, i.e. assume $\langle R\cup L\rangle \ne \mathbb {P}^m$.
Since $\deg (R\cap L) = 1$, we get $e_1 \ge d+2+d+1-1$, a contradiction. Now assume $m=3$ (and hence $w_0\le 2d+4$) and $L\cap R  =\emptyset$. We
have $\deg (W_0\cap (L\cup R)) \ge 2d+3$ and $A_1\nsubseteq L\cup R$. Hence one of the lines $L, R$ (call it $D$ and call $J$ the other one) is spanned by the only degree two
subscheme of $A$ and $B\cup S\subset L\cup R$. Since $A_1$ is curvilinear and $L\cup R$ is the scheme-theoretic base locus of $|\mathcal {I}_{L\cup R}(2)|$,
there is a quadric surface $Q\supset L\cup R$ such that $A_1\nsubseteq Q$. Since $\deg (\mbox{Res}_Q(W_0)) =1$,
we have $h^1(\mathcal {I}_{\mbox{Res}_Q(W_0)}(d-2)) =0$. By \cite[Lemma 5.1]{bb2} applied to the degree two surface $Q$ instead of a plane
we get $A_1\subset Q$, a contradiction.

\quad (b) Assume $g=1$. Since $W_0\cap H_1$ spans $H_1$, we have $e_1\ge d+m+1$ (Remark \ref{w3.0}).

\quad (b1) Assume $h^1(\mathcal {I}_{W_1}(d-1)) >0$. Therefore $\deg (W_1) \ge d+1$ and hence $e_1\le d+2m-3$. Since $(d+m+1) +2d > 2d+2m-2$, we have
$\deg (W_1) <2d$. Hence there is a line $L\subset \mathbb {P}^m$ such that $\deg (W_1\cap L) \ge d+1$. Hence $e_1\le d+2m-3 \le 2d+1$.
Therefore there is a line $R\subset H_1$ such that $\deg (R\cap W_0)\le d+2$. Since $\deg (A\cap L) \le 2$, we have $L\cap (B\setminus B\cap H_1) \ne \emptyset$.
Since $R\subset H_1$, we have $R\ne L$. As in step (a1.1) we get a contradiction.

\quad (b2) Assume $h^1(\mathcal {I}_{W_1}(d-1)) =0$. Copy step (a1.2.1.2) with $H_1$ instead of $M$.\end{proof}

\begin{remark}\label{e1e}
Fix integers $b>0$ and $m\ge b-1$. Let $E\subset \mathbb {P}^m$ be a curvilinear connected scheme of degree $b$. Set $\{O\}:= E_{red}$.
The local Hilbert scheme of the regular local ring $\mathcal {O}_{\mathbb {P}^m,O}$ is smooth and irreducible at $E$ and of dimension
$(m-1)(b-1)$ (\cite{g}). Therefore, if we don't prescribe $O$, but only the integer $b$ we get that the set of all possible $E$ is parametrized
by an irreducible variety of dimension $m+(m-1)(b-1)$. Alternatively, we may consider the equisingular deformations
of $E$ in $\mathbb {P}^m$ and get (in the curvilinear case) that it is parametrized by a smooth variety
of dimension $m+(m-1)(b-1)$. If $d\ge b-1$, then $\dim (\langle \nu _d(E)\rangle )=b-1$. If we take a type $(s;b_1,\dots ,b_s)$ with $\sum _ib _i =5$
and general $E_1,\dots ,E_s$ with $E_i$ curvilinear and $\deg (E_i)=b_i$, then in the computation of Remark \ref{e0e} below
we only need to use that $\dim (\langle \nu _d(E_1\cup \cdots \cup E_s)\rangle )=4$ and get the dimensions claimed in Remark \ref{e0e} and Proposition \ref{zzz1} below.
\end{remark}

\begin{remark}\label{zzz2}
Fix a combinatorial type $(s;b_1,\dots ,b_s)$ with $s\ge 1$ and $b_1+\cdots +b_s=5$. All $A$ evincing the the border rank of some degree $d \ge 9$ polynomial
and with combinatorial type $(s;b_1,\dots ,s)$ is obtained in the following way. Fix a $4$-dimensional linear subspace $M\subseteq \mathbb {P}^m$ and $s$
linearly independent points $O_1,\dots ,O_s\in M$. For $i=1,\ldots ,s$ let $M_i\subseteq M$ be a $(b_i-1)$-dimensional linear subspace such that $O_i\in M_i$, $\dim (M_i) =b_i$
and $M_1\cup \cdots \cup M_s$ spans $M$, i.e. the subspaces $M_1,\dots ,M_s$ are linearly independent. Let $D_i\subseteq M_i$ be the rational normal curve of $M_i$
with the convention that $D_i=\{O_i\}$ if $b_i=1$ (i.e. if $M_i =\{O_i\}$) and $D_i =M_i$ if $b_i=2$ (i.e. if $M_i$ is a line). Let $A_i\subset D_i$ be the effective degree $b_i$
divisor of $D_i$ with $O_i$ as its support with the convention $A_i =\{O_i\}$ if $b_i=1$. Set $A:= A_1\cup \cdots \cup A_s$. Since each $A_i$ spans $M_i$, $A$
is a curvilinear degree $5$ scheme of combinatorial type $(s;b_1,\dots ,b_s)$ spanning $M$. Since $d\ge 4$, we have $\dim \langle \nu _d(A)\rangle =4$. Take any $P\in \langle \nu _d(A)\rangle$ such that $P\notin \langle \nu _d(A')\rangle$ for any $A'\subsetneq A$. Since $d\ge 5$, $P$ has border rank $5$ and cactus rank $5$ and $A$ is the
only scheme evincing the cactus rank of $P$ (\cite[Corollary 2.7]{bgl}). All possible $M, O_1,\dots ,O_s,D_1,\dots ,D_s$ are projectively equivalent and so all schemes
with combinational type $(s;b_1,\dots ,b_s)$ are projectively equivalent.
\end{remark}

\begin{remark}\label{e0e}
Since $d\ge 5$, a theorem of Alexander and Hirschowitz gives that $\sigma _5(X_{m,d})$ and $\sigma _5(X_{m,d})\setminus \sigma _4(X_{m,d})$ have dimension
$5m+4$. Now assume $d\ge 9$ so that each point of $\sigma _5(X_{m,d})\setminus \sigma _4(X_{m,d})$ has a unique type.
For any type $(s;b_1,\dots ,b_s)$ let $\sigma _5(X_{m,d})(s;b_1,\dots ,b_s)$ denote the set of all $P\in \sigma _5(X_{m,d})\setminus \sigma _4(X_{m,d})$
with type $(s;b_1,\dots ,b_s)$. Let $\sigma _5(X_{m,d})(s;b_1,\dots ,b_s)'$ denote the subset of the algebraic set $\sigma _5(X_{m,d})(s;b_1,\dots ,b_s)$ formed by all $P$ whose
cactus rank is evinced by a linearly independent scheme. Let $\sigma _5(X_{m,d})(s;b_1,\dots ,b_s)''$ denote the subset of $\sigma _5(X_{m,d})(s;b_1,\dots ,b_s)$ formed by all $P$ whose
cactus rank is evinced by a curvilinear scheme. In the proof of Theorem \ref{i2} we checked that $\sigma _5(X_{m,d})(s;b_1,\dots ,b_s)''\subseteq \sigma _5(X_{m,d})(s;b_1,\dots ,b_s)'$. By a theorem of Chevalley on the image of an algebraic variety by an algebraic map (\cite[Ex. I.3.18]{h}) 
all sets $\sigma _5(X_{m,d})(s;b_1,\dots ,b_s)$, $\sigma _5(X_{m,d})(s;b_1,\dots ,b_s)''$, $\sigma _5(X_{m,d})(s;b_1,\dots ,b_s)'$ are constructible. Thus we may speak about the irreducible components of the sets
$\sigma _5(X_{m,d})(s;b_1,\dots ,b_s)$, $\sigma _5(X_{m,d})(s;b_1,\dots ,b_s)''$, $\sigma _5(X_{m,d})(s;b_1,\dots ,b_s)'$ and the dimension of each of this irreducible component; indeed, they are the irreducible components of the closure (in the Zariski topology) of the subsets $\sigma _5(X_{m,d})(s;b_1,\dots ,b_s)$ or $\sigma _5(X_{m,d})(s;b_1,\dots ,b_s)'$
of the projective variety
$\sigma _5(X_{m,d})$. Take two schemes $A_1,A_2$ with the same type. Since $\deg (A_1\cup A_2)\le 10 \le d-1$, we have $h^1(\mathcal {I}_{A_1\cup A_2}(d))=0$.
Hence Grassmann's formula gives $\langle \nu _d(A_1)\rangle \cap \langle \nu _d(A_2)\rangle =\langle \nu _d(A_1\cap A_2)\rangle$. Using the description of all possible
schemes of type $(s;b_1,\dots ,b_s)$ given in the proof of Theorem \ref{i2} and Remark \ref{e1e} we get 
that each set $\sigma _5(X_{m,d})(s;b_1,\dots ,b_s)''$ is irreducible and with the following dimension:
\begin{itemize}
\item[(i)] $\sigma _5(X_{m,d})(1;5)''$ has dimension $5m$.
\item[(ii)] $\sigma _5(X_{m,d})(2;3,2)''$ has dimension $5m+1$.
\item[(iii)] $\sigma _5(X_{m,d})(2;4,1)''$ has dimension $5m+1$.
\item[(iv)] $\sigma _5(X_{m,d})(3;3,1,1)''$ has dimension $5m+2$.
\item[(v)] $\sigma _5(X_{m,d})(3;2,2,1)''$ has dimension $5m+2$.
\item[(vi)] $\sigma _5(X_{m,d})(4;2,1,1,1)''$ has dimension $5m+3$.
\item[(vii)] $\sigma _5(X_{m,d})(5;1,1,1,1)''$ has dimension $5m+4$.
\end{itemize} 
We get that $\sigma _5(X_{m,d})(s;b_1,\dots ,b_s)''$ has codimension $5-s$ in $\sigma _5(X_{m,d})$ (this is the expected codimension). Since each
$\sigma _5(X_{m,d})(s;b_1,\dots ,b_s)''$ is irreducible, we get that $\sigma _5(X_{m,d})(s;b_1,\dots ,b_s)'$ is irreducible of dimension $5m+s-1$ and
that it contains a non-empty Zariski open subset of $\sigma _5(X_{m,d})(s;b_1,\dots ,b_s)''$, proving the following result.
\end{remark}
\begin{proposition}\label{zzz1}
For each $(s;b_1,\dots ,b_s)$ the sets $\sigma _5(X_{m,d})(s;b_1,\dots ,b_s)'$ and $\sigma _5(X_{m,d})(s;b_1,\dots ,b_s)''$ are irreducible and of dimension
$5m+s-1$.
\end{proposition}
To get the dimension of the the set of all associated degree $d$ homogeneous polynomials, add $+1$ to each dimension, because proportional polynomials,
say $f$ and $cf$ with $c\in \mathbb {C}\setminus \{O\}$,
have the same associated point $P\in \mathbb {P}^r$

\section{Remarks on possible further works}\label{S+}

\begin{remark}\label{e2e}
We briefly explain what is missing to complete the stratification by rank of $\sigma _5(X_{m,d})\setminus \sigma _4(X_{m,d})$. The case with very low $d$ seems
to be difficult, full of nasty cases and not very enlightening, but cases like $d=7,8$ may be easy (when $A$ is not unique, just take the one which gives a lower rank). It may be easy
to check which integers $r_{m,d}( P)$ appears for some $P\in \sigma _5(X_{m,d})\setminus \sigma _4(X_{m,d})$ just because only very few integers
are possible. For a fixed $P$ the proof of Theorem \ref{i2} may be not enough to compute the integer $r_{m,d}( P)$, but probably all a priori possible integers arise for
some $P$ for which the proof of Theorem \ref{i2} works.

Fix any integer $d\ge 5$ and any $A$ evincing the cactus rank of some $P$. If $A$ is curvilinear, but not connected, i.e. if $s\ge 2$, then we know its single pieces $A_i$ and in each case by \cite{bgi} or \cite{bb2} we know the rank, $\rho _i$, of any  $Q_i\in \langle \nu _d(A_i))$ such that $Q_i\notin \langle \nu _d(A')\rangle$ for any $A'\subsetneq A_i$. We always have $r_{m,d}( P) \le \sum _{i=1}^{s} \rho _i$. This is a good
upper bound (e.g. we always have $r_{m,d}( P) \le 3d+1$ if $s=2$ and $r_{m,d}( P) \le 2d+1$ if $s=3$), but in some cases we know that it is not
the exact value. The case $s=5$ is trivial and hence
we assume $s\le 4$. If $\dim (\langle A\rangle )=1$, then we have $r_{m,d}( P) =d-3$ by a theorem of Sylvester (\cite{cs}). If $\dim (\langle A\rangle )=2$ (and in particular always if $m=2$), then
an upper bound is obtained
looking at the minimal degree of a reduced curve containing $A$ (see Proposition \ref{t1}).
When $s=1$ and $\dim (\langle A\rangle )=3$ there is a unique non-curvilinear case (\cite[Case III of Theorem 1.3]{eh}). For this case we have a lower bound, $3d-3$, and an upper bound, $4d$, but we don't even know if the integer $r_{m,d}( P)$ is the same for all non-curvilinear $A$ and all $P$
with $A$ as its associated scheme.
\end{remark}

\begin{proposition}\label{t1}
Assume $d\ge 9$. Fix $P\in \mathbb {P}^r$ with $b_{r,m}( P) =5$ and assume the existence of a plane $H\subset \mathbb {P}^m$ such that $P\in \langle \nu _d(H)\rangle$.
Then $r_{m,d}( P) \le 3d$. Let $A$ be the only scheme evincing the border rank of $A$. If either $h^0(H,\mathcal {I}_A(2)) \ge 2$ or $A$ is contained in a reduced conic, then
$r_{m,d}( P) \le 2d$.
\end{proposition}

\begin{proof}
By concision (\cite[Exercise 3.2.2.2]{l}) we may assume $m =2$. Since $d\ge 9$, there is a unique zero-dimensional
scheme $A\subset \mathbb {P}^2$ with $\deg (A) = 5$ and evincing the scheme rank and the border rank of $P$. Since $h^0(\mathcal {O}_{\mathbb {P}^2}(2)) =6$, there is a conic $C\supset A$. If $C$ is reduced, then $r_{m,d}( P) \le 2d$, because $C$ is connected, $\dim (\langle \nu_d( C)\rangle )=2d$
and \cite[Proposition 5.1]{lt} holds (with the same proof) for reduced and connected non-degenerate curves. Hence we may assume that $C$ is a double line, say $C=2L$.
It is easy to check that $\deg (L\cap A) \ge 3$. If $\deg (L\cap A)\ge 4$, then $\mbox{Res}_L(A)$ is a single point, $O$, and hence $A$ is contained in any conic $L\cup D$
with $D$ a line containing $O$ (we use that $L$ and $D$ are Cartier divisors of $\mathbb {P}^m$). Hence we may assume $\deg (L\cap A)=3$. In this case we have $A\subset L\cup D_1\cup D_2$
with $D_1$ and $D_2$ any two distinct lines, $\ne L$, and containing the scheme $\mbox{Res}_L(A)$. We get $r_{m,d}( P) \le 3d$ quoting again the same modified version
of \cite[Proposition 5.1]{lt}. We have $h^0(\mathcal {I}_A(2)) \ge 2$ if and only if there is a line $L$ with $\deg (L\cap A)\ge 4$.
\end{proof}

\section{Conclusions}
We give the stratification by the symmetric tensor rank of all degree $d \ge 9$ homogeneous polynomials
with border rank $5$ and which depend essentially on at least 5 variables, extending two previous works (\cite{bgi} for border
ranks 2 and 3, \cite{bb2} for border rank 4). For the polynomials
which depend on at least 5 variables only 5 ranks are possible: $5$, $d+3$, $2d+1$, $3d-1$, $4d-3$, but each of the ranks
$3d-1$ and $2d+1$ is achieved in two geometrically different situations. These ranks are uniquely determined by a certain degree
5 zero-dimensional scheme $A$ associated to the polynomial. The polynomial depends essentially on at least 5 variables
if and only if $A$ is linearly independent. The polynomial has rank $4d-3$ (resp $3d-1$, resp. $2d+1$, resp. $d+3$, resp. $5$)
if $A$ has $1$ (resp. $2$, resp. $3$, resp. $4$, resp. $5$) connected components. The assumption $d\ge 9$ guarantees that each polynomial
has a uniquely determined associated scheme $A$. 

In each case we describe the dimension of the families of the
polynomials with prescribed rank, each irreducible family being determined by the degrees
of the connected components of the associated scheme $A$. Each family of polynomials has dimension $5m+s$, where $s=1,\dots ,5$
is the number of connected components of $A$ (Remark \ref{e0e}). 

When $A$ has at least $3$ connected components we also describe
all linear forms evincing the rank $\rho$ of the polynomial $f$, i.e. all linear forms $\ell _i$ such that $f =\sum _{i=1}^{\rho} \lambda _i\ell _i^d$, $\lambda _i\in \mathbb {C}\setminus \{0\}$, up to
a non-zero multiple of each $\ell _i$ and a permutation of $\ell _1,\dots ,\ell _c$.

The proofs require projective geometry and some algebraic geometry.

\begin{acknowledgements}
We thanks the referee for useful comments.
\end{acknowledgements}



\begin{thebibliography}{}
\bibitem{ACCF} L. Albera, P. Chevalier, P. Comon and A. Ferreol: On the virtual array concept for higher order array processing, IEEE Trans. Sig. Proc. 53, no. 4, 1254--1271 (2005)













\bibitem{b} E. Ballico: Subsets of the variety $X\subset \mathbb {P}^n$ evincing the $X$-rank of a point of $\mathbb {P}^n$, Houston J. Math. 42, no. 3, 803--824 (2016)


\bibitem{bb1} E. Ballico and A. Bernardi: Decomposition of homogeneous polynomials with low rank ,
Math. Z. 271, 1141--1149 (2012)

\bibitem{bb2} E. Ballico and A.  Bernardi: Stratification of the fourth 
secant variety of Veronese variety via the symmetric rank, 
Adv. Pure Appl. Math. 4, no. 2, 215--250  (2013)


\bibitem{bgi} A. Bernardi, A. Gimigliano and M. Id\`{a}: Computing symmetric rank for 
symmetric tensors, J. Symbolic Comput. 46, no. 1, 34--53 (2011)

\bibitem{br} A. Bernardi and K. Ranestad: The cactus rank of cubic forms, J. Symbolic Comput. 50, 291--297 (2013)






\bibitem{bcmt} J. Brachat, P. Comon, B. Mourrain and E. P. Tsigaridas: Symmetric tensor decomposition, 
Linear Algebra Appl. 433(11-12), 1851--1872 (2010)


\bibitem{bb+} W. Buczy\'{n}ska and J. Buczy\'{n}ski: Secant varieties to high degree veronese reembeddings, catalecticant matrices and smoothable Gorenstein schemes, J. Algebraic Geom. 23, 63--90 (2014)



\bibitem{bgl} J. Buczy\'{n}ski, A. Ginensky and J. M. Landsberg: 
Determinantal equations for secant varieties and the 
Eisenbud-Koh-Stillman conjecture,
J. London Math. Soc. (2) 88, 1--24 (2013)

\bibitem{Ch} P. Chevalier: Optimal separation of independent narrow-band sources - concept and performance, Signal Processing, Elsevier, 73, 27--48, special issue on blind separation and deconvolution (1999)

\bibitem{cs} G. Comas and M. Seiguer: On the rank of a binary form.  Found. Comp. Math. 11, no. 1, 65--78 (2011)

\bibitem{Co1} P. Comon: Independent Component Analysis, in: J-L. Lacoume, editor, Higher Order Statistics, 29--38. Elsevier, Amsterdam, London, (1992)

\bibitem{cglm} P. Comon, G. H. Golub, L.-H. Lim and B. Mourrain:
Symmetric tensors and symmetric tensor rank, 
SIAM J. Matrix Anal.  30, no. 3,1254--1279 (2008)

\bibitem{cm} P. Comon and B. Mourrain: Decomposition of quantics in sums of powers of linear forms, Signal Processing, Elsevier 53, 2, (1996)


\bibitem{dLC} L. De Lathauwer and J. Castaing: Tensor-based techniques for the blind separation of ds-cdma signals, Signal Processing 87, no. 2, 322--336 (2007)

\bibitem{DM} M. C. Dog$\check{\mathrm{a}}$n and J. M. Mendel: Applications of cumulants to array processing. I. aperture extension and array calibration, IEEE Trans. Sig. Proc. 43, no. 5, 1200--1216 (1995)

\bibitem{eh} D. Eisenbud and J. Harris: Finite projective schemes in linearly general position,  J. Algebraic Geom.  1,  no. 1, 15--30  (1992)

\bibitem{ep} Ph. Ellia and  Ch. Peskine: Groupes de points de ${\bf {P}}^2$: caract\`{e}re et position uniforme, in: Algebraic geometry (L' Aquila, 1988), 111--116,
Lecture Notes in Math., 1417, Springer, Berlin (1990)


\bibitem{g} M. Granger: G\'{e}om\'etrie des sch\'{e}mas de Hilbert ponctuels, M\'{e}m. Soc. Math. France (N.S.) $2^e$ s\'erie 8, 1--84  (1983)

\bibitem{h} R. Hartshorne: Algebraic geometry, Springer-Verlag, Berlin (1977)

\bibitem{j} J. Jelisiejew: An upper bound for the Waring rank of a form, Arch. Math. 102, 329--336 (2014)

\bibitem{JS} T. Jiang and  N. D. Sidiropoulos: Kruskal's permutation lemma and the identification of CANDECOMP/PARAFAC and bilinear models, IEEE Trans. Sig. Proc. 52,
no. 9, 2625--2636 (2004)

\bibitem{ik} A. Iarrobino and V. Kanev:
Power sums, Gorenstein algebras, and determinantal loci,
Lecture Notes in
Mathematics, vol. 1721, Springer-Verlag, Berlin, Appendix C by 
Iarrobino and Steven L. Kleiman (1999)



\bibitem{l} J. M. Landsberg: Tensors: Geometry and Applications,
Graduate Studies in Mathematics, Vol. 128, Amer. Math. Soc. Providence (2012)


\bibitem{lt} J. M. Landsberg and Z. Teitler: On the ranks and border ranks of symmetric tensors, Found. Comput. Math. 10, no. 3, 339--366 (2010)


\bibitem{ls} L.-H. Lim and V. de Silva: Tensor rank and the ill-posedness of the best low-rank approximation problem, SIAM J. Matrix Anal. Appl. 30, no. 3, 1084--1127 (2008)

\bibitem{McC} P. McCullagh: Tensor Methods in Statistics.  Monographs on Statistics and Applied Probability, Chapman and Hall (1987)

\bibitem{rs} K. Ranestad and F.-O. Schereyer: On the rank of a symmetric form, J. Algebra 346, 340--342 (2011)

\end{thebibliography}
\end{document}